\documentclass[reqno,15pt]{amsart}
\usepackage{amsmath}
\usepackage{amssymb}
\usepackage{amscd}
\usepackage{graphics}
\usepackage{latexsym}
\usepackage{hyperref}

\theoremstyle{plain}
\newtheorem{theorem}{Theorem}[section]
\newtheorem{thm}[theorem]{Theorem}
\newtheorem{cor}[theorem]{Corollary}
\newtheorem{lem}[theorem]{Lemma}
\newtheorem{prop}[theorem]{Proposition}

\newcounter{kludge}
\newcounter{kludgeb}
\theoremstyle{definition}
\newtheorem{defn}[theorem]{Definition}

\newtheorem{ques}[theorem]{Question}

\newtheorem{rmk}[theorem]{Remark}

\newtheorem{notat}[theorem]{Notation}

\newtheorem{conj}[theorem]{Conjecture}

\theoremstyle{remark}

\input xy
\xyoption{all}

\newcommand{\marpar}[1]{}
\newcommand{\mni}{\medskip\noindent}

\newcommand{\mbb}{\mathbb}
\newcommand{\QQ}{\mbb{Q}}

\newcommand{\ZZ}{\mbb{Z}}
\newcommand{\CC}{\mbb{C}}
\newcommand{\RR}{\mbb{R}}

\newcommand{\PP}{\mbb{P}}

\newcommand{\mc}{\mathcal}
\newcommand{\mcA}{\mc{A}}

\newcommand{\mcF}{\mc{F}}

\newcommand{\mcL}{\mc{L}}

\newcommand{\mcY}{\mc{Y}}

\newcommand{\mf}{\mathfrak}

\newcommand{\mfM}{\mathfrak{M}}
\newcommand{\fm}[1]{\mfM_{#1}}

\newcommand{\OO}{\mc{O}}

\newcommand{\wt}{\widetilde}

\newcommand{\ol}{\overline}
\newcommand{\ul}{\underline}

\newcommand{\M}{\overline{\mc{M}}}
\newcommand{\Mo}{\mc{M}}

\newcommand{\Kgn}[1]{\Mo_{#1}}

\newcommand{\Kgnb}[1]{\M_{#1}}

\newcommand{\So}{S}

\newcommand{\Ll}{\mathcal{L}}
\newcommand{\Cc}{C}

\newcommand{\mm}{X}
\newcommand{\nn}{Y}
\newcommand{\oom}{\omega}
\newcommand{\jj}{J}
\newcommand{\eb}{\beta}
\newcommand{\NEE}{{\text{NE}}}
\newcommand{\NEB}{\overline{\NEE}}
\newcommand{\nee}[2]{\NEE_{#1}^{#2}}
\newcommand{\neb}[2]{\NEB_{#1}^{#2}}
\newcommand{\nef}[1]{\nee{\text{f}}{#1}}

\newcommand{\ner}[1]{\nee{\text{r}}{#1}}

\newcommand{\snef}[1]{\nee{\text{s.f}}{#1}}

\newcommand{\sner}[1]{\nee{\text{s.r}}{#1}}

\newcommand{\snep}[1]{\nee{\text{s.p.f}}{#1}}

\newcommand{\lt}{\left}
\newcommand{\rt}{\right}

\newsavebox{\sembox}
\newlength{\semwidth}
\newlength{\boxwidth}

\newsavebox{\semrbox}
\newlength{\semrwidth}
\newlength{\boxrwidth}

\title
{Symplectic Invariance of Rational Surfaces on K\"{a}hler Manifolds} 

\author[Starr]{Jason Michael Starr}
\address{Department of Mathematics \\
  Stony Brook University \\ Stony Brook, NY 11794}
\email{jstarr@math.stonybrook.edu} 

\date{\today}
\begin{document}


\begin{abstract}
  Koll\'{a}r and Ruan proved symplectic deformation invariance for
  uniruledness of K\"{a}hler manifolds.  Zhiyu Tian proved the same
  for rational connectedness in dimension $\leq 3$.  Koll\'{a}r
  conjectured this in all dimensions.  We prove Koll\'{a}r's
  conjecture, as well as existence of a covering family of rational
  surfaces, for all K\"{a}hler manifolds that are symplectically
  deformation equivalent to $G/P$ or to a low degree complete
  intersection in such.
\end{abstract}


\maketitle



\section{Statement of Results} \label{sec-sot}  
\marpar{sec-sot}

\mni
Theorem \ref{thm-PG} produces rational curves and surfaces via
Gromov-Witten theory for $G/P$.  Theorem \ref{thm-CI} extends this to
complete intersections.

\begin{ques} \label{ques-first} \marpar{ques-first}
  For a connected, compact manifold $\mm$ with (integrable) complex
  structure $\jj$ and with the symplectic form $\oom$ of a K\"{a}hler
  metric, which holomorphic properties depend only on $\oom$?  Which
  depend only on the deformation class of $\oom$?
\end{ques}

\mni
The (topological) complex vector bundle $T^{1,0}_{\mm,\oom}$
underlying the holomorphic tangent bundle $T^{1,0}_{\mm,\jj}$ depends
only on the deformation class, as do its Chern classes.  Gromov-Witten
invariants also depend only on the deformation class.

\begin{thm}\cite[Theorem 4.2.10]{KLow},
  \cite[Proposition 4.9]{RuanVirt}  \label{thm-KR1} \marpar{thm-KR1} 
  For connected, compact, K\"{a}hler manifolds, uniruledness is
  invariant under symplectic deformation.
\end{thm}

\begin{conj}[Koll\'{a}r] \label{conj-K} \marpar{conj-K}
  For connected, compact, K\"{a}hler manifolds, rational connectedness
  is invariant under symplectic deformation.
\end{conj}

\mni
The best result on this conjecture is a theorem of Zhiyu Tian.

\begin{thm}[Zhiyu Tian] \cite{Zhiyu} \label{thm-ZT} \marpar{thm-ZT}
  In complex dimension $\leq 3$, rational connectedness is preserved
  by symplectic deformation equivalence.
\end{thm}

\mni
The formulation below has no explicit mention of Gromov-Witten invariants.

\begin{thm}[Symplectic rational curves and surfaces on
  $G/P$] \label{thm-PG} \marpar{thm-PG}
  \textbf{1.}
  Every generalized complex flag variety $G/P$ is a fiber type Fano
  manifold: the Mori cone is a free $\ZZ_{\geq 0}$-semigroup generated
  by extremal rays $\RR_{> 0}\cdot \eb_i$ of classes $\eb_i$ of free
  rational curves.

  \noindent
  \textbf{2.}
  Koll\'{a}r's conjecture holds for $G/P$: every K\"{a}hler manifold
  $Z$ that is symplectically deformation equivalent to $G/P$ is
  rationally connected.

  \noindent
  \textbf{3.}
  Finally, $Z$ has a covering family of rational surfaces, except if
  $Z \cong \CC\PP^1$.
\end{thm}

\mni
The basic notions of Mori theory depend sensitively on $\jj$ rather
than $\oom$.  

\begin{defn}[Mori cone] \label{defn-Moricone} \marpar{defn-Moricone}
  The \textbf{Mori cone}, $\neb{}{\jj}(\mm)_{\RR}$, is the closure in
  $H_2(\mm;\RR)$ of the $\RR_{\geq 0}$-semigroup
  $\nee{}{\jj}(\mm)_{\RR}$ spanned by classes of $\jj$-holomorphic
  curves.  These classes are $\jj$-\textbf{effective}, and the
  $\ZZ_{\geq 0}$-semigroup spanned by these classes is
  $\nee{}{\jj}(\mm)$.

  \noindent
  A class $\eb\in \nee{}{\jj}(\mm)$ is $\jj$-\textbf{extremal} if
  $0<\langle c_1(T^{1,0}_{\mm,\oom}), \eb \rangle \leq 1 +
  \text{dim}_{\CC}(\mm)$ and if, for some K\"{a}hler
  $[D]\in H^2(\mm;\RR)$, the function
  $$
  \neb{}{\jj}(\mm)_{\RR}\setminus \{0\} \to \RR,\ \ \alpha \mapsto
  \langle c_1(T^{1,0}_{\mm,\oom})-[D],\alpha \rangle 
  /
  \langle [D],\alpha \rangle,
  $$
  has a nonnegative maximum precisely on the ray $\RR_{>0}\cdot \eb$.

  \noindent
  An element of a $\ZZ_{\geq 0}$-semigroup $N$ in $H_2(\mm;\ZZ)$ is
  $\jj$-\textbf{reducible} if it equals a nonzero sum of elements of
  $N$ plus a nonzero sum of classes of $\jj$-holomorphic spheres.
  Otherwise, it is $\jj$-\textbf{irreducible}.  A nonzero element of
  $N$ is \textbf{decomposable} if it equals a sum of two nonzero
  classes in $N$. Otherwise, it is \textbf{indecomposable}.
\end{defn}

\mni
Our focus is exclusively on extremal rays whose associated contraction
is fiber type.  This occurs if and only if the ray is generated by a
$\jj$-extremal class that is \emph{free}.

\begin{defn}[Free classes] \label{defn-free} \marpar{defn-free}
  A non-constant $\jj$-holomorphic map $u:\CC\PP^1\to \mm$ is
  $\jj$-\textbf{free}, resp. $\jj$-\textbf{very free}, if the
  holomorphic vector bundle $u^*T^{1,0}_{\mm,\jj}$ on $\CC\PP^1$ is
  semi-ample, resp. ample.  Then, the class $u_*[\CC\PP^1]$ is
  $\jj$-\textbf{free}, resp.  $\jj$-\textbf{very free}.

  \noindent
  Denote the $\ZZ_{\geq 0}$-span of such classes by $\nef{\jj}(\mm)$.
  \noindent
  A $\jj$-extremal class is a $\jj$-\textbf{ruling class} if the
  associated contraction is a ruling by conics.  These classes
  generate the $\ZZ_{\geq 0}$-semigroup $\ner{\jj}(\mm)$.
\end{defn}

\mni
The next theorem uses symplectic ``freeness'' via Gromov-Witten
invariants (for $\oom$).

\begin{defn}[Symplectically free classes] \label{defn-sfree}
  \marpar{defn-sfree}
  A class $\eb\in H_2(\mm;\ZZ)$ is \textbf{symplectically effective in
    genus} $g$ if there exists $n\in \ZZ_{\geq 0}$ such that the full
  Gromov-Witten invariant is not identically zero,
  $$
  \text{GW}^{\mm,\oom}_{g,n,\eb}:
  H^*(\mfM_{g,n};\QQ) \otimes H^*(\mm^{n};\QQ)\to \QQ.
  $$

  \noindent
  Denote by $\eta_\mm$ the Poincar\'{e} dual of the homology class of
  a point.

  \noindent
  The class $\eb$ is \textbf{symplectically pseudo-free} if there
  exists $n\in \ZZ_{\geq 0}$, if there exists
  $u\in H^*(\mfM_{0,n+1};\QQ)$, and if there exists
  $v\in H^*(\mm^n;\QQ)$ with
  $GW^{\mm,\oom}_{0,n+1,\eb}(u\otimes w)\neq 0$ for
  $w:=\text{pr}_1^*\eta_{\mm}\smile \text{pr}_2^*v \in H^*(\mm\times
  \mm^n;\QQ)$.

  \noindent
  Denote the $\ZZ_{\geq 0}$-span of such classes by
  $\snep{\oom}(\mm)$.
  
  \noindent
  The \textbf{normal degree} is
  $m_\eb=m_\eb(\mm):= \langle c_1(T^{1,0}_{\mm,\oom}), \eb \rangle -
  2.$ If $\eb$ is free, then $m_\eb\geq 0$.

  \noindent
  For $m_\eb\geq 0$, the \textbf{first gravitational descendant} is
  $$
  f_\eb = f_\eb(\mm) := \langle \tau_{m_\eb}(\eta_{\mm}) \rangle^{\mm,\oom}_{0,\eb}
  = \text{GW}^{\mm,\oom}_{0,1,\eb}(c_1(\psi)^{m_\eb}\otimes \eta_\mm). 
  $$
  If $0\leq m_\eb\leq \text{dim}_{\CC}(\mm)-1$ and if $f_\eb>0$, then
  $\eb$ is \textbf{symplectically free}.

  \noindent
  Denote the $\ZZ_{\geq 0}$-span of such classes by
  $\snef{\oom}(\mm)$.

  \noindent
  If $m_\eb =0$, if $f_\eb=1$, and if there exists a symplectically
  pseudo-free class that is $\QQ$-linearly independent from $\eb$,
  then $\eb$ is \textbf{symplectically ruling}.

  \noindent
  Denote the $\ZZ_{\geq 0}$-span of such classes by
  $\sner{\oom}(\mm)$.

  \noindent
  If $m_\eb(\mm)>0$ and $f_\eb(\mm)>0$, then the \textbf{second
    gravitational descendant} and \textbf{quotient invariant} are
  $$
  s_{\eb} = s_{\eb}(\mm) := \lt \langle \tau_{m_\eb-1}(\eta_m),
  \text{ch}_2(T^{1,0}_{\mm,\oom}) +
  \frac{m_\eb}{2(m_\eb+2)^2}c_1(T^{1,0}_{\mm,\oom})^2
  \rt\rangle^{\mm,\oom}_{0,\eb}, \ \ q_{\eb}(\mm) := \frac{s_\eb}{f_\eb}.
  $$
  A symplectically free class $\eb$ is \textbf{symplectically
    $2$-free} if $m_\eb>0$ and $s_\eb>0$.
\end{defn}

\mni
The relevance of $f_\eb$ and $s_\eb$ is a sharp version of the Theorem
\ref{thm-KR1}.  Koll\'{a}r and Ruan prove that for every $\jj$-free
$\eb\in \nef{\jj}(\mm)$ that is $\jj$-irreducible, the evaluation map,
$$
\text{ev}:\Kgnb{0,1}((\mm,\jj),\eb) \to \mm,
$$
has fiber $F=\text{ev}^{-1}(\{q\})$ that is smooth, projective, with
pure dimension $m_\eb$ and parameterizing only maps with smooth domain
for every $q$ contained in a dense ``Zariski'' open
$\mm_\eb\subset \mm$.  Thus, the $\jj$-irreducible $\jj$-free classes
equal the $\jj$-irreducible symplectically pseudo-free classes,
$\eb \in \snep{\oom}(\mm)$.

\begin{thm}[Relative ampleness of $\psi$] \label{thm-simpFano}
  \marpar{thm-simpFano}
  \textbf{1.}
  The $\jj$-irreducible $\jj$-free classes also equal the
  $\jj$-irreducible symplectically free classes,
  $\eb \in \snef{\oom}(\mm)$, and these exist if and only if
  $\snef{\oom}(\mm)$ is nonzero.

  \noindent
  \textbf{2.}
  For every such class $\eb$, the $\psi$-line bundle is ample on the
  fiber $F$ of $\text{ev}$ over every point of $\mm_\eb$.  If
  $m_\eb>0$, the two leading coefficients of the Hilbert polynomial
  are,
  $$
  \chi(F,\psi^{\otimes d}) = \frac{f_\eb}{m_\eb!}d^{m_\eb} +
  \frac{s_\eb}{2(m_\eb-1)!} d^{m_\eb-1} + \dots =
  \frac{f_\eb d^{m_\eb}}{m_\eb!}\lt( 1 + \frac{m_\eb q_\eb}{2d} +
  \dots \rt).
  $$

  \noindent
  \textbf{3.}
  For every $[D_0],\dots,[D_r]\in H^2(\mm;\RR)$, there is an equality
  of Gysin classes,
  $$
  \pi_*u^*(D_0\smile \dots\smile D_r) =
  \langle D_0,\eb \rangle \cdots \langle D_r,\eb \rangle
  c_1(\psi)^r \in H^{2r}(F_{\eb,x};\RR), \text{ thus}
  $$
  $$
  \langle \tau_s(\eta_\mm), D_0\smile \dots \smile D_r, -
  \rangle^{\mm,\oom}_{0,\eb} =
  \langle D_0,\eb \rangle \cdots \langle D_r,\eb \rangle
  \langle \tau_{r+s}(\eta_\mm),-\rangle^{\mm,\oom}_{0,\eb}.  
  $$
  In particular, $\mm$ is uniruled if and only if $\snef{\oom}(\mm)$
  is nonzero.
\end{thm}

\mni
The general form of Theorem \ref{thm-PG} is as follows.

\begin{thm}[Rational curves and surfaces via
  $\snef{\oom}(\mm)$] \label{thm-surf} \marpar{thm-surf}
  \textbf{1.}
  For K\"{a}hler $\mm$ and $Z$ that are symplectically deformation
  equivalent, both are rationally connected whenever
  $\snef{\oom}(\mm)$ is nonzero and the $\QQ$-span of
  $\snep{\oom}(\mm)$ equals $H_2(\mm;\QQ)$.

  \noindent
  \textbf{2.}
  A K\"{a}hler $Z$ that is symplectically equivalent,
  resp. symplectically deformation equivalent, to $\mm$ is covered by
  rational surfaces if $\snef{\oom}(\mm)$ is nonzero and if at least
  one nonzero $\oom$-minimal class, resp. every nonzero indecomposable
  class, in $\snef{\oom}(\mm)$ is symplectically ruling or
  symplectically $2$-free.

  \noindent
  \textbf{3.}
  Conversely, for a $\jj$-irreducible $\eb\in\snef{\oom}(\mm)$ with
  $m_\eb>0$, if a general $\text{ev}$-fiber $F$ has
  $c_1(T_F)\in H^2(F;\RR)$ nonzero and pseudo-effective, then $\eb$ is
  symplectically $2$-free.
\end{thm}

\begin{ques}
  \label{ques-GW} \marpar{ques-GW}
  Does every rationally connected $\mm$ have a nonzero $2$-point
  invariant $\text{GW}^{\mm,\oom}_{0,n+2,\eb}(u\otimes w)$ for some
  $\eb\in H_2(\mm;\ZZ)$, for some $u\in H^*(\mfM_{0,n+2};\QQ)$, and
  for some $v\in H^*(\mm^n;\QQ)$, where
  $w:=\text{pr}_1^*(\eta_{\mm\times \mm})\smile \text{pr}_2^*(v)\in
  H^*(\mm^2\times \mm^n;\QQ)$?
\end{ques}

\begin{rmk} \label{rmk-surf} \marpar{rmk-surf}
  Every Hirzebruch surface $\mm=\Sigma_{n}$, satisfies (1) above.
  However, for $n\geq 4$, for every $\eb$ such that the 2-point
  evaluation map to $\mm\times \mm$ is dominant, every fiber has
  excess irreducible components.  Thus the \emph{proof method} of
  Theorem \ref{thm-KR1} does not imply nonzero $2$-point invariants.
\end{rmk}

\mni
Larger than the class of generalized complex flag varieties $G/P$ are
the \emph{smooth complete intersections},
$\mm~=~Y_1\cap \dots \cap Y_c$, of analytic hypersurfaces $Y_i$ in
$G/P$.

\begin{ques} \label{ques-second} \marpar{ques-second}
  Which complete intersections $\mm$ in $G/P$ are fiber type Fano?
  Which have the above property of existence of a covering family of
  rational surfaces?
\end{ques}

\begin{notat}[Invariants of complete intersections] \label{notat-rel}
  \marpar{notat-rel}
  For K\"{a}hler $\nn$, for every smooth, codimension-$c$ intersection
  $\mm$ of nef, analytic hypersurfaces $\nn_1,\dots,\nn_c$ in $\nn$,
  for every $\jj$-effective curve class $\eb$ of $Y$, denote
  $$
  m_{\eb}(\nn,\mm):= \sum_{j=1}^c \langle [\nn_j],\eb \rangle, \ \
  f_{\eb}(\nn,\mm):= \prod_{j=1}^c\lt(\langle [\nn_j],\eb\rangle\rt)!,
  $$
  $$
  q_{\eb}(\nn,\mm):= \sum_{j=1}^c \frac{\langle [Y_j],\eb \rangle\lt(
  1+ \langle [Y_j],\eb \rangle\rt)}{2}, \ \ 
  s_{\eb}(Y,\mm):= q_{\eb}(\nn,\mm) f_{\eb}(Y).
  $$
\end{notat}

\begin{rmk} \label{rmk-ruling} \marpar{rmk-ruling}
  For a smooth analytic hypersurface $\mm$ in a compact, K\"{a}hler
  manifold $\nn$, for a \emph{ruling class} $\eb$ on $Y$ that is a
  pushforward of an extremal class on $\mm$, if
  $\langle [\mm],\eb \rangle >0$, then $\eb$ is not fiber type on
  $\mm$.  We need to \emph{contract} the ruling classes.
\end{rmk}

\mni
\begin{defn}[Contraction of ruling classes] \label{defn-qtt} \marpar{defn-qtt}
  A \textbf{contraction of ruling classes} on $\nn$ is a surjective
  holomorphic map, $\pi:\nn \to \nn'$, with $\nn'$ normal of Fujiki
  class $C$ such that every fiber $\nn_q=\pi^{-1}(\{q\})$ is connected
  and the image of $H_2(\nn_q;\ZZ)$ in $H_2(\nn;\ZZ)$ equals the
  $\ZZ$-span of $\ner{\jj}(Y)$.
\end{defn}

\begin{defn} \label{defn-simpFano} \marpar{defn-simpFano}
  A K\"{a}hler $\mm$ is \textbf{integrally fiber type Fano} if the
  primitive generator in $H_2(\mm;\ZZ)$ of every extremal ray of
  $\nef{\jj}(\mm)_{\RR}$ is $\jj$-free.  It is \textbf{simplicially
    Fano} if also $\nef{\jj}(\mm)$ is a \emph{free}
  $\ZZ_{\geq 0}$-semigroup,
  $\prod_i \ZZ_{\geq 0}\cdot \eb_i \cong \ZZ_{\geq 0}^{b_2(\mm)}$.
\end{defn}

\begin{thm}[Complete intersections] \label{thm-CI} \marpar{thm-CI}
  \textbf{1.}
  Assume that the compact, K\"{a}hler manifold $\nn$ has a submersive
  contraction of ruling classes, $\pi:\nn\to \nn'$.  Let
  $\mm=\pi^{-1}(\mm')$ be the inverse image of a complete intersection
  $\mm'=\nn'_1\cap \dots \cap \nn'_c$ of positive, complex analytic
  hypersurfaces $\nn'_i\subset \nn'$ with
  $\text{dim}_{\CC}(\mm')\geq 3$.

  \noindent
  \textbf{2.}
  Assume further that $\nn$ is integrally fiber type Fano,
  resp. simplicially Fano.  Then also $\mm$ is integrally fiber type
  Fano, resp. simplicially Fano, and the pushforward map induces an
  isomorphism $\nef{\jj}(\mm) \to \nef{\jj}(\nn)$ if and only if
  $m_{\eb_i}(\nn,\mm)\leq m_{\eb_i}(\nn)$ for every primitive
  generator $\eb_i\in \nef{\jj}(Y)$ of an extremal ray.  In this case,
  there are identities
  $$
  m_{\eb_i}(\mm)
  =m_{\eb_i}(\nn)-m_{\eb_i}(\nn,\mm), \ \ f_{\eb_i}(\mm)
  =f_{\eb_i}(\nn)\cdot f_{\eb_i}(\nn,\mm).
  $$

  \noindent
  \textbf{3.}
  For $\mm$ and $\nn$ satisfying the conditions in 1 and 2, a
  K\"{a}hler $Z$ that is symplectically equivalent to $\mm$,
  resp. symplectically deformation equivalent to $\mm$, is covered by
  rational surfaces if at least one $\oom$-minimal
  $\eb_i\in \snef{\oom}(\nn)$, resp.  every indecomposable
  $\eb_i\in \snef{\oom}(\nn)$, satisfies either
  \begin{enumerate}
  \item[(i)] $\eb_i$ is symplectically ruling, or
  \item[(ii)] $\eb_i$ is symplectically $2$-free,
  $m_{\eb_i}(\nn,\mm) < m_{\eb_i}(\nn)$, and
  $s_{\eb_i}(\nn,\mm)<s_{\eb_i}(\nn)$.
  \end{enumerate}
  In Case (i), $\eb_i$ is symplectically ruling for $\mm$.  In Case
  (ii), $\eb_i$ is symplectically $2$-free for $\mm$ and
  $$
  q_{\eb_i}(\mm) = q_{\eb_i}(\nn) -
  q_{\eb_i}(\nn,\mm), \ \ 
  s_{\eb_i}(\mm) = f_{\eb_i}(\nn,\mm)\cdot\lt( s_{\eb_i}(\nn) -
  s_{\eb_i}(\nn,\mm)\rt).
  $$
\end{thm}

\begin{prop}[$\snef{\oom}(Y)$ and fiber type Fano
  manifolds] \label{prop-fiber} \marpar{prop-fiber}   
  \textbf{1.}
  A K\"{a}hler $\mm$ is fiber type Fano if and only if
  $\snef{\oom}(\mm)_{\RR}$ is dual to the closure of the K\"{a}hler
  cone.

  \noindent
  \textbf{2.}
  Then it is integrally fiber type Fano if and only if a general fiber
  of the contraction of each primitive, $\jj$-extremal $\eb_i$ has
  Fano index equal to the pseudo-index, $m_{\eb_i}+2$.
\end{prop}

\begin{ques} \label{ques-simpFano} \marpar{ques-simpFano}
  Is a Fano manifold simplicial whenever it is integrally fiber type?
\end{ques}

\mni
All evidence points to a positive answer.  I am grateful to Cinzia
Casagrande who taught me about positive results in this direction.
Wi{\'s}niewski bounds the number of extremal rays of fiber type Fano
manifolds, \cite{Wis91}.  Druel proves that when this bound is
attained, then the Mori cone is simplicial, \cite{Druel}.  Ou studies
this problem when the number of extremal rays is one smaller that the
maximal bound, \cite{Ou}.  Finally, Casagrande proves a positive
answer whenever the complex dimension is $\leq 5$, \cite{Casagrande}.


\section{Approaches to uniruledness, rational connectedness, and   
  rational surfaces} \label{sec-uni}   
\marpar{sec-uni}      

\mni
Gromov-Witten invariants of an almost complex, symplectic manifold
rely on a \emph{virtual structure}.  This is a trace of transversality
on the solution space of the $\overline{\partial}$-equation defining
$\jj$-holomorphic curves.  This virtual structure is independent of
deformations of $(\jj,\oom)$.

\mni
The Koll\'{a}r-Ruan Theorem establishes transversality of the space of
$\jj$-holomorphic spheres containing a general point of a K\"{a}hler
manifold when the free curve class is $\jj$-irreducible, i.e., when
``bubbling'' does not occur.  Thus, a K\"{a}hler manifold is uniruled
if and only if there exists a symplectically pseudo-free curve class.
Therefore, uniruledness is independent of deformations of
$(\jj,\oom)$.

\mni
Unfortunately, transversality fails for $\jj$-holomorphic spheres
containing two or more general points.  Zhiyu Tian's approach to
Koll\'{a}r's conjecture compensates for non-transversality via an
explicit study of the Mori program applied to a K\"{a}hler $\mm$ that
is symplectically deformation equivalent to a rationally connected,
projective manifold.  Since divisorial and flipping contractions are
birational transformations, the work of Hu-Li-Ruan on symplectic
birational cobordism invariance of Gromov-Witten invariants plays a
key role, \cite{HuLiRuan}.  Zhiyu Tian's tour-de-force proof of
Koll\'{a}r's conjecture in dimension $3$ uses both Gromov-Witten
theory and the explicit classification of threefold extremal
contractions by Mori and Mukai, \cite{MoriMukai}.

\mni
There are several theorems giving alternative characterizations and
extensions of uniruledness with applications to other questions in
symplectic topology.  For instance, for symplectic manifolds with a
Hamiltonian $\mathbb{S}^1$-action (an analogue of K\"{a}hler manifolds
with a nontrivial holomorphic $\CC^\times$-action), McDuff uses Seidel
elements to prove uniruledness via a ring-theoretic property of the
quantum cohomology ring, \cite{McDuff}.  Also, McLean proves that for
smooth affine varieties, uniruledness of projective compactifications
is a symplectic invariant of the underlying symplectic manifold of the
smooth affine variety, \cite{McLean}.

\mni
Following Voisin, \cite{VoisinRC}, and Zhiyu Tian, \cite{Zhiyu}, our
approach to Koll\'{a}r's conjecture instead relies on the maximally
rationally connected fibration.  For a compact, K\"{a}hler manifold
$(\mm,\jj,\oom)$, denote the Barlet space by $\text{Barlet}(\mm,\jj)$,
cf. \cite{Barlet}.

\begin{defn} \label{defn-frq} \marpar{defn-frq}
  A \textbf{Zariski open subset} of $\mm$ is the open complement
  $\mm^o$ of a closed complex analytic subspace of $\mm$.

  \noindent
  A \textbf{meromorphic function} on $\mm$ is an equivalence class of
  pairs $(\mm^o,\phi)$ of a dense Zariski open subset $\mm^o$ of $\mm$
  and a holomorphic function,
  $$
  \mm \supseteq \mm^o \xrightarrow{\phi} Q^o \subseteq Q,
  $$
  to a dense Zariski open subset $Q^o$ of a compact complex analytic
  space $Q$ in Fujiki class $C$ such that the closure of the graph of
  $\phi$ is a complex analytic subspace of $\mm\times Q$.

  \noindent
  An \textbf{almost holomorphic fibration} of $\mm$ is a finite (and
  proper) holomorphic map with normal domain,
  $$
  \Phi:Q\to \text{Barlet}(\mm,\jj),
  $$
  such that the pullback universal cycle in $\mm\times Q$ is the
  closure of the graph of a meromorphic function,
  $$
  \phi:\mm\supseteq \mm^o\to Q^o\subseteq Q,
  $$
  that is a proper holomorphic submersion between dense Zariski opens
  $\mm^o$ of $\mm$ and $Q^o$ of $Q$.

  \noindent
  A \textbf{rational quotient} of $(\mm,\jj)$ is an almost holomorphic
  fibration with the following property.  For every pair
  $$
  (\pi:\Cc \to M, u:\Cc \to \mm),
  $$
  of (1) a proper, flat, holomorphic map $\pi$ to a normal, connected
  analytic space $M$ whose fibers are trees of rational curves and (2)
  a holomorphic map $u$ such that the composition $\phi\circ u$ is
  defined and dominant, there exists a dense, Zariski open
  $M^o\subset M$ and a commutative diagram of complex analytic spaces,
  $$
  \begin{CD}
  C^o @> u^o >> \mm^o \\
  @V \pi^o VV  @VV \phi V \\
  M^o @>> v^o > Q^o
  \end{CD},
  $$
  where $C^o$ equals $\pi^{-1}(M^o)$.
\end{defn}

\begin{rmk} \label{rmk-frq} \marpar{rmk-frq}
  The existence of the rational quotient of each connected, compact
  K\"{a}hler manifold was proved in \cite{Ca}.  When $(\mm,\jj)$ is a
  complex projective manifold, this is the same notion as the
  \textbf{maximally rationally connected fibration,} which was
  constructed for normal projective schemes over arbitrary fields
  (including positive characteristic fields) in \cite{KMM}.
\end{rmk}

\mni
Since the fiber dimension of $\phi$ is at least $m_\eb+1$ for every
$\jj$-irreducible $\eb\in \snef{\oom}(\mm)$, the fiber dimension is at
least the maximum of all $m_\eb$.

\begin{notat} \label{notat-mmax} \marpar{notat-mmax}
  Denote by $m$, resp. by $m'$, the maximum, resp. minimum, of $m_\eb$
  for every $\jj$-irreducible $\eb\in \snef{\oom}(\mm)$.
\end{notat}

\mni
The pullback by $\phi$ (rather, the corresponding graph closure
considered as a correspondence) is a $\ZZ$-linear map of cohomology
that is compatible with the pullback map on global sections of
reflexive coherent analytic sheaves,
$$
H^{\ell}(Q;\ZZ) \to H^{\ell}(\mm;\ZZ), \text{ resp. }
H^0(Q,(\Omega_{Q/\CC}^{\otimes \ell})^{\vee \vee}) \to
H^0(\mm,\Omega_{\mm/\CC}^{\otimes \ell}).
$$

\begin{prop} \label{prop-ratqtt} \marpar{prop-ratqtt}
  A compact, connected, K\"{a}hler manifold $\mm$ is rationally
  connected if for every non-uniruled $Q$ as above with
  $1\leq \text{dim}_{\CC}(Q) \leq \dim_{\CC}(\mm)-m-1$, for at least
  one $\ell$ positive,
  $h^0(Q,(\Omega_{Q/\CC}^{\otimes \ell})^{\vee
    \vee})>h^0(\mm,\Omega_{\mm/\CC}^{\otimes \ell})$.
\end{prop}

\begin{proof}
  By \cite{GHS}, the quotient $Q$ is either a point or it is
  non-uniruled of positive dimension.  Since $\phi$ is a holomorphic
  submersion on a dense open $\mm^o$, the pullback map on holomorphic
  contravariant tensors is injective.
\end{proof}

\mni
This is particularly useful if the dimension inequalities imply that
$Q$ has small dimension, e.g., if $Q$ is a curve.  In this case,
$h^{\ell,0}(Q)$ is positive for some $\ell>0$. If $b_\ell(\mm)\leq 1$,
this contradicts the Hodge inequality $b_\ell \geq 2h^{\ell,0}$.
Conjecturally, every non-uniruled $Q$ of positive dimension has
positive $h^0(Q,(\Omega_{Q/\CC}^{\otimes \ell})^{\vee\vee})$ for all
$\ell \gg 0$.  The variant of this technique proved at the beginning
of Section \ref{sec-Mori} is even simpler, and it is not contingent
upon conjectures.


\subsection{Rational surfaces} \label{ssec-surf}
\marpar{ssec-surf}      

\mni
There is no transversality for the spaces of holomorphic rational
surfaces in a K\"{a}hler manifold: the $\overline{\partial}$-equation
is over-determined.  However, the virtual structure on the space of
$\jj$-holomorphic spheres determines a \emph{virtual first Chern
  class} on each fiber $F$ of the evaluation map,
$$
\text{ev}:\Kgnb{0,1}((\mm,\jj),\eb) \to \mm.
$$

\mni
This is computed in \cite{dJS3}.  If $F$ is pure of complex dimension
$m_\eb$, the virtual first Chern class equals the first Chern class of
the dual of the cotangent complex of $F$ (which is perfect of
amplitude $[-1,0]$).  When the fiber $F$ is also smooth, this virtual
first Chern class equals the actual first Chern class of $F$.

\mni
The formula from \cite{dJS3} for the virtual first Chern class is a
pullback of an element of $H^*(\mf{M}_{0,1};\QQ)\otimes H^*(\mm;\QQ)$.
Thus, for each element
$\gamma\in H^*(\mf{M}_{0,1};\QQ)\otimes H^*(\mm;\QQ)$ in complementary
degree $2(m_\eb-1)$, the Gromov-Witten invariant of the cup product of
the virtual first Chern class and $\gamma$ is a gravitational
descendant.  Thus, it is a symplectic deformation invariant.

\mni
When the pullback of $\gamma$ is the Poincar\'{e} dual cohomology
class of the curve class of a moving family of curves in $F$, this
gravitational descendant equals the anticanonical degree of a moving
family of curves in $F$.  Whenever this degree is positive, $F$ is
uniruled by Mori's theorem.  Finally, for a big divisor class
$[D]\in H^2(F;\QQ)$, the class $\gamma = [D]^{m_\eb-1}$ is a
$\QQ_{>0}$-multiple of the Poincar\'{e} dual of a moving curve class
on $F$.  Thus, $F$ is uniruled if the gravitational descendant is
positive for some big divisor class $[D]$.  This is the essence of our
approach: find enough pairs $(\eb,[D])$ of an indecomposable
$\eb\in \snef{\oom}(\mm)$ and an element
$[D]\in H^2(\mf{M}_{0,1}\times \mm;\QQ)$ with positive gravitational
descendant so that for every K\"{a}hler structure $(J,\oom)$ in the
deformation class, there exists at least one such pair with $\eb$ a
$\jj$-irreducible class and with $[D]$ a big class on $F$.

\mni
The earlier work, \cite{dJS9}, followed a similar strategy to prove
uniruledness of \emph{all} spaces $\Kgnb{0,0}((\mm,\jj),\eb)$, not
merely those with $\eb$ a $\jj$-irreducible element of
$\snef{\oom}(\mm)$.  In that case, the class $[D]$ is a ``quasi-map
divisor class''.  The pairing of the virtual first Chern class with
the appropriate power $[D]^n$ is positive if
$c_1(T^{1,0}_{\mm,\oom})^2$, resp.  $\text{ch}_2(T^{1,0}_{\mm,\oom})$,
has positive pairing, resp. nonnegative pairing, with the homology
class of every $\jj$-holomorphic, irreducible surface in $(\mm,\jj)$.
The K\"{a}hler $\mm$ is \textbf{Fano} if $c_1(T^{1,0}_{\mm,\oom})$ is
positive, and it is \textbf{weakly $2$-Fano} if also
$\text{ch}_2(T^{1,0}_{\mm,\oom})$ has nonnegative pairing with
$\jj$-holomorphic, irreducible surfaces.

\begin{thm}\cite{dJS9} \label{thm-2Fano} \marpar{thm-2Fano}
  For every weakly $2$-Fano manifold $\mm$ with pseudo-index $\geq 3$,
  free rational curves parameterized by an irreducible component $M_i$
  of $\Kgnb{0,0}((\mm,\jj),\eb)$ move in pencils on a covering family
  of rational surfaces provided that every component $M_j$
  intersecting $M_i$ parameterizes some maps with irreducible domain.
\end{thm}

\mni
There are many examples of such $\mm$.  The best classification
results are due to Araujo and Castravet, \cite{AC}, \cite{AC2}.  They
also prove a sharper version of this theorem.  Although the
transversality hypothesis on $M_j$ is valid in many cases, there are
also cases where it fails.  Moreover, the cone of $\jj$-effective
homology classes of surfaces is certainly not a symplectic deformation
invariant.  Thus, neither the hypotheses nor conclusion of this
theorem are invariant under deformation of $\jj$.

\mni
Our method here produces a covering family of rational surfaces from
positivity of gravitational descendants, precisely by restricting to
indecomposable $\eb\in \snef{\oom}(\mm)$ and by replacing the
quasi-map divisor $[D]$ by the psi class $\psi$.


\subsection{Next steps} \label{ssec-next}
\marpar{ssec-next}      

\mni
This approach is the first step in a larger program to construct
``very twisting'' ruled surfaces on K\"{a}hler manifolds with
sufficient positivity of $c_1(T^{1,0}_{\mm,\oom})$ and
$\text{ch}_2(T^{1,0}_{\mm,\oom})$.  These surfaces come from covering
families of rational curves in parameter spaces
$\Kgnb{0,1}((\mm,\jj),\eb)$ for $\jj$-irreducible classes $\eb$, just
as the rational surfaces produced in this paper.  However, the very
twisting ruled surfaces also satisfy additional constraints.  By
\cite{dJHS} and \cite{Zhu}, very twisting surfaces are the key
ingredient in proving existence of rational sections for fibrations
over a surface whose general fiber is deformation equivalent to
$(\mm,\jj)$.  By \cite{dJS10}, cf. also \cite{DeLand} and
\cite{Minoccheri}, existence of very twisting surfaces also implies
the Weak Approximation Conjecture of Hassett and Tschinkel for
$(\mm,\jj)$, cf. \cite{HT06} and \cite{HassettWA2}.  In each earlier
paper, \cite{dJHS}, \cite{Zhu}, \cite{DeLand}, and \cite{Minoccheri},
the construction of each very twisting surface uses special properties
of the particular K\"{a}hler manifold $(\mm,\jj,\oom)$.  The goal here
is to construct covering ruled rational surfaces in the most robust
manner possible: using Gromov-Witten invariants rather than special
properties.

\mni
Finally, a leading open case for existence of very twisting ruled
surfaces is when $(\mm,\jj,\oom)$ is the wonderful compactification
orbifold $\widehat{G}$ of a simply connected semisimple complex Lie
group.  Existence of very twisting ruled surfaces in this case is the
key missing ingredient in a type-free proof of Serre's ``Conjecture
II'' over function fields for all simply connected, semisimple
algebraic groups (not merely the split and quasi-split groups).  The
generalized complex flag manifolds appear as closed orbits in the
orbit decomposition of $\widehat{G}$.  It would be useful to either
compute directly the relevant gravitational descendants $f_\eb$ and
$s_\eb$ for $\widehat{G}$, or to use virtual localization to reduce
this to the computation of gravitational descendants on those
generalized complex flag manifolds appearing as closed orbits in
$\widehat{G}$.


\subsection{Changes from previous versions} \label{ssec-changes}
\marpar{ssec-changes}

\mni
A previous version of this paper included a description of the
``quasi-map divisor class'' $[D]$, as well as extensions of the main
theorems to include the case when the big divisor class on the
$\text{ev}$-fiber $F$ is the quasi-map class rather than the psi
class.  As mentioned, those extensions are not invariant under
deformation of the almost complex structure $\jj$, both because
transversality of the moduli space of $\jj$-holomorphic curves is not
invariant, and because the cone of $\jj$-effective surfaces is not
invariant.  This version focuses on consequences of positivity of the
psi class and results that are invariant under deformation of $\jj$.


\section{Gromov-Witten invariants} \label{sec-GW}
\marpar{sec-GW}

\mni
The numbers $m_\eb$, $f_\eb$, and $s_\eb$ are defined via
Gromov-Witten invariants.  Here is a quick review of the relevant
aspects of Gromov-Witten invariants for compact, K\"{a}hler manifolds.

\mni
For integers $g,n\geq 0$, denote by $\fm{g,n}$ both the Artin
$\CC$-stack (for $\CC$-schemes with the fppf topology) and the
associated complex analytic stack parameterizing flat, proper families
of $n$-pointed curves
$$
(\Cc,(p_1,\dots,p_n)),
$$
that are \textbf{prestable}, i.e., 
$\Cc$ is a proper, connected, reduced, at-worst-nodal curve 
of arithmetic genus $g$ together with an ordered $n$-tuple 
$(p_1,\dots,p_n)$ 
of distinct, smooth points of $\Cc$.  
For every homology
class $\eb\in H_2(\mm;\ZZ)$, the
\textbf{virtual dimension} of the space of genus-$g$, $n$-pointed
stable maps of class $\eb$ is 
$$
d=d^{\mm,\oom}_{g,n,\eb} = \langle c_1(T^{1,0}_{\mm,\oom}), \eb
\rangle + (\text{dim}_\CC(\mm)-3)(1-g) + n.
$$
For every connected, closed, symplectic manifold $(\mm,\oom)$, the
\textbf{Gromov-Witten invariant} is a $\QQ$-linear functional on
cohomology,
$$
\text{GW}^{\mm,\oom}_{g,n,\eb}:H^{2d}(\fm{g,n}\times \mm^n;\QQ)\to \QQ.
$$

\mni
For $n\geq 1$, for every $i=1,\dots,n$, denote by $\psi_i$ the
$\text{i}^{\text{th}}$ \textbf{psi class}, i.e., the relative
dualizing sheaf of the $1$-morphism forgetting the $i^\text{th}$
marked point,
$$
\pi_i:\fm{g,n}\to \fm{g,n-1}.
$$

\mni
For every $n\geq 0$, for every $n$-tuple of nonnegative integers,
$\ul{m}=(m_1,\dots,m_n)$, for every $n$-tuple
$(\gamma_1,\dots,\gamma_n)\in H^*(\mm,\QQ)^n$ of homogeneous classes
of degrees $\text{deg}(\gamma_i)=e_i$ with $2|\ul{m}| + |\ul{e}|$
equal to $2d$, the associated \textbf{gravitational descendant} equals
$$
\langle \tau_{m_1}(\gamma_1),\dots,\tau_{m_n}(\gamma_n)
\rangle^{\mm,\oom}_{g,\eb} := 
\text{GW}^{\mm,
  \oom}_{g,n,\eb}\lt( \prod_{i=1}^n q^*c_1(\psi_i)^{m_i}\smile
\text{pr}_i^*\gamma_i\rt).
$$
In particular, 
$$
f_\eb= \langle \tau_{m_\eb}(\eta_M) \rangle^{\mm,\oom}_{0,\eb},
\ \
s_\eb = \langle \tau_{m_\eb-1}(\eta_M),
\text{ch}_2(T^{1,0}_{\mm,\oom})
+  
\frac{m_\eb}{2(m_\eb+2)^2}c_1(T^{1,0}_{\mm,\oom})
\rangle^{\mm,\oom}_{0,\eb}. 
$$

\mni
For a connected, closed, symplectic manifold $(\mm,\oom,\jj)$ with an
$\oom$-tame almost complex structure $\jj$, one construction of the
Gromov-Witten invariant uses the $\QQ$-linear functional of pairing
against a \emph{virtual fundamental class} (in homology) for elements
of the cohomology of a moduli space $\Kgnb{g,n}((\mm,\jj),\eb)$.  Here
$\Kgnb{g,n}((\mm,\jj),\eb)$ is the moduli stack parameterizing
genus-$g$, $n$-pointed $\jj$-holomorphic \textbf{stable maps} of class
$\eb$ to $\mm$,
$$
(\Cc,(p_1,\dots,p_r),u:\Cc\to \mm).
$$
The datum $(\Cc,(p_1,\dots,p_r))$ is an object of $\fm{g,n}$. The map
$u$ is $\jj$-holomorphic with pushforward class $u_*[\Cc] = \eb$.
Finally, the datum is \textbf{stable}, i.e., the log dualizing sheaf
$\omega_{\Cc}(\underline{p}_1+\dots+\underline{p}_r)$ on $\Cc$ is
ample on every irreducible component $\Cc_i$ such that $u_*[\Cc_i]$
vanishes.

\mni
The \textbf{forgetful map} associates to each stable map the
underlying object of $\fm{g,n}$,
$$
\Phi:\Kgnb{g,n}((\mm,\jj),\eb) \to \fm{g,n}, \ \
(\Cc,(p_1,\dots,p_r),u:\Cc\to \mm) \mapsto (\Cc,(p_1,\dots,p_r)).
$$
Similarly, the \textbf{evaluation map} is the map
$$
\text{ev}:\Kgnb{g,n}((\mm,\jj),\eb)\to \mm^n, \ \
(\Cc,(p_1,\dots,p_r),u:\Cc\to \mm) \mapsto (u(p_1),\dots,u(p_n)).
$$

\mni
There are natural embeddings of $\Kgnb{g,n}((\mm,\jj),\eb)$ into
(infinite-dimensional) function spaces of stable $L^p_1$ maps that are
not necessarily $\jj$-holomorphic, cf. \cite{LiTian} and
\cite{Zinger}, and this gives $\Kgnb{g,n}((\mm,\jj),\eb)$ a
topological and metric structure for which $\Phi$ and $\text{ev}$ are
continuous.  When $(\mm,\oom,\jj)$ is K\"{a}hler, then
$\Kgnb{g,n}((\mm,\jj),\eb)$ is even a complex analytic Deligne-Mumford
stack whose coarse moduli space is in Fujiki class $C$ and compact.
In particular, it has finitely many irreducible components, each of
which is a compact complex analytic space in Fujiki class $C$.

\mni
Li and Tian give an analytic construction of the virtual fundamental
class in the homology of $\Kgnb{g,n}((\mm,\jj),\eb)$.  There is a
construction within algebraic geometry if $\mm$ is projective using
perfect obstruction theories.  That construction also computes the
tangent bundle on the smooth locus of the moduli space by \cite{dJS3}.
The algebraic construction applies for classes $\eb$ that are
$\text{ev}$-dominant, since the rational quotient is relatively
projective,
$$
\phi:\mm\supseteq \mm^o\to Q^o\subseteq Q.
$$

\begin{defn} \label{defn-dom} \marpar{defn-dom}
  A class $\eb\in H_2(\mm;\ZZ)$ is $\text{ev}$-\textbf{dominant} if
  the associated evaluation map is surjective,
  $$
  \text{ev}_{0,1,\eb}:\Kgnb{0,1}((\mm,\jj),\eb)\to \mm.
  $$
  An irreducible component of $\Kgnb{0,0}((\mm,\jj),\eb)$ is
  $\text{ev}$-\textbf{dominant} if the inverse image in
  $\Kgnb{0,1}((\mm,\jj),\eb)$ surjects to $\mm$ under the evaluation
  map.
\end{defn}

\begin{prop} \label{prop-frq} \marpar{prop-frq}
  \textbf{1.}
  A connected, compact, K\"{a}hler manifold is projective if it is
  rationally connected, e.g., this holds for the fibers of $\phi$.

  \noindent
  \textbf{2.}
  For every $\eb\in H_2(\mm;\ZZ)$, shrinking $Q^o$ and $\mm^o$ if
  necessary, there exists a Zariski open $\Kgnb{0,0}((\mm,\jj),\eb)^o$
  of $\Kgnb{0,0}((\mm,\jj),\eb)$ whose inverse image
  $\Kgnb{0,1}((\mm,\jj),\eb)$ in $\Kgnb{0,1}((\mm,\jj),\eb)$ equals
  $\text{ev}^{-1}(\mm^o)$ and such that there is a commutative
  diagram,
  $$
  \begin{CD}
    \Kgnb{0,1}((\mm,\jj),\eb)^o @>>> \mm^o \\
    @VVV  @VVV \\
    \Kgnb{0,0}((\mm,\jj),\eb)^o @>>> Q^o
  \end{CD},
  $$
  that identifies $\Kgnb{0,0}((\mm,\jj),\eb)^o$ with the relative
  moduli space $\Kgnb{0,0}(\mm^o/Q^o,\eb)$.

  \noindent
  \textbf{3.}
  The $Q^o$-fibers of $\Kgnb{0,0}((\mm,\jj),\eb)^o$ are algebraic
  Deligne-Mumford stacks whose coarse moduli spaces are complex
  projective varieties.  Up to shrinking $Q^o$ further, both
  $\Kgnb{0,0}((\mm,\jj),\eb)^o$ and the coarse moduli space are flat
  and projective over $Q^o$.
\end{prop}

\begin{proof}
  The fibers of $\phi$ are connected, compact K\"{a}hler manifolds
  that are rationally connected, where the restriction of $\oom$ is a
  K\"{a}hler class in $H^2_{\RR}$.  For every such manifold, all small
  perturbations of the K\"{a}hler class in $H^2_{\RR}$ are symplectic
  classes.  Every rationally connected manifold has vanishing
  $h^{\ell,0}$ for every $\ell>0$, so that $H^2_\CC$ equals $H^{1,1}$.
  Thus, sufficiently small perturbations of the K\"{a}hler class are
  K\"{a}hler $(1,1)$-classes in $H^2_\RR$.  Since $H^2_\QQ$ is dense
  in $H^2_\RR$, some of these K\"{a}hler $(1,1)$-classes are rational.
  Therefore, every rationally connected, K\"{a}hler manifold is a
  complex projective manifold by the Kodaira embedding theorem.

\mni
By the defining property of the rational quotient, up to shrinking
$Q^o$, every $\eb$-curve that intersects $\mm^o$ is contained in a
fiber of $\phi$.  Thus, the relative moduli space equals the open
subset $\Kgnb{0,0}((\mm,\jj),\eb)^o$ parameterizing maps that
intersect $\mm^o$, i.e., the Zariski open complement of the closed
analytic subspace of $\Kgnb{0,0}(\mm\setminus \mm^o,\eb)$ of
$\Kgnb{0,0}(\mm,\eb)$.

\mni
Since $\mm^o$ is relatively projective over $Q^o$, also
$\Kgnb{0,0}(\mm^o/Q^o,\eb)$ is $Q^o$-relatively an algebraic
Deligne-Mumford stack whose coarse moduli space is relatively
projective.  By generic flatness, up to shrinking $Q^o$, these are
even $Q^o$-flat.
\end{proof}

\mni
Because $\Kgnb{0,0}(\mm^o/Q^o,\eb)$ is an algebraic Deligne-Mumford
stack over $Q^o$ whose coarse moduli space is projective, and both are
flat over $Q^o$, the algebraic results of \cite{dJS3} apply.  This
gives a formula for the canonical divisor class on the maximal open
that has pure dimension $d$.


\section{Ampleness of the Psi Class} 
\label{sec-psi}  \marpar{sec-psi}

\mni
One of the fundamental facts about the evaluation map for genus $0$
stable maps is that obstructedness of the evaluation map is
independent of the choice of the marked point, at least when the
domain is irreducible.  This follows from facts about vanishing of
higher cohomology of coherent sheaves on a genus $0$ curve.  The
fastest way to formulate these facts is via the ``universal
extension'' of the structure by the relative dualizing sheaf; roughly
just the direct sum of two copies of the dual Serre twisting sheaf
$\mathcal{O}(-1)$.

\begin{defn}\cite[Section 3]{dJS3} \label{defn-Q} \marpar{defn-Q}
  For every proper, flat map of complex analytic spaces,
  $\pi:\Cc\to \So,$ whose fibers are reduced, at-worst-nodal,
  connected, genus-$0$ curves, the \textbf{universal extension} of the
  relative dualizing sheaf $\omega_\pi$ is the short exact sequence of
  coherent analytic sheaves on $\Cc$,
  $$
  \Xi_\pi: \ \ 0\to \omega_\pi \to E_\pi \to \OO_{\Cc} \to 0,
  $$
  unique up to $\OO_{\So}^\times$ and compatible with arbitrary base
  change, such that the coherent analytic sheaves $\pi_*E_\pi$ and
  $R^1\pi_*E_\pi$ are both zero.  For every finite rank, locally free
  $\OO_{\Cc}$-module $\mcF$, for the associated short exact sequence
  $$
  \textit{Hom}_{\OO_{\Cc}}(\mcF,\Xi_\pi): \ \ 0\to
  \textit{Hom}_{\OO_{\Cc}}(\mcF,\omega_\pi) \to
  \textit{Hom}_{\OO_{\Cc}}(\mcF,E_\pi)  \to
  \textit{Hom}_{\OO_{\Cc}}(\mcF,\OO_{\Cc}) \to 0, 
  $$
  the \textbf{adjoint map} is the connecting map of higher direct
  image sheaves,
  $$
  \delta_{\Xi,\mcF}:\pi_*\mcF^\vee \to
  R^1\pi_* \textit{Hom}_{\OO_{\Cc}}(\mcF,\omega_\pi),
  $$
  whose fiber map at each $s\in\So$, via Serre duality, equals
  $$
  \delta_{\Xi,\mcF,s}: (\pi_*\mcF^\vee)_s \to H^0(\Cc_s,\mcF_s)^\vee.
  $$
\end{defn}

\mni
Up to taking duals, vanishing of a higher direct image sheaf involving
the universal extension allows us to encode both vanishing of higher
direct image of our original sheaf, and also surjectivity of the
adjoint map.

\begin{lem} \label{lem-Q} \marpar{lem-Q}
  For every finite rank, locally free $\OO_{\Cc}$-module $\mcF$, the
  higher direct image sheaf
  $R^1\pi_*\textit{Hom}_{\OO_{\Cc}}(\mcF,E_\pi)$ is zero if and only
  if both $R^1\pi_*(\mcF^\vee)$ is zero and for every $s\in \So$, the
  adjoint map $\delta_{\Xi,\mcF,s}$ is surjective.
\end{lem}

\begin{proof}
  Since $R^q\pi_*$ of a coherent sheaf vanishes for every $q\geq 2$,
  the long exact sequence of higher direct images associated to the
  short exact sequence $\textit{Hom}_{\OO_{\Cc}}(\mcF,\Xi_\pi)$ yields
  $$
  \pi_*(\mcF^\vee) \xrightarrow{\delta} R^1\pi_*(\mcF^\vee\otimes
  \omega_\pi) \to R^1\pi_*\textit{Hom}_{\OO_{\Cc}}(\mcF,E_\pi) \to
  R^1\pi_*(\mcF^\vee) \to 0.
  $$
  Thus, $R^1\pi_*\textit{Hom}_{\OO_{\Cc}}(\mcF,E_\pi)$ vanishes if and
  only if $R^1\pi_*(\mcF^\vee)$ vanishes and $\delta_{\Xi,\mcF}$ is
  surjective.  Finally, $\delta_{\Xi,\mcF}$ is surjective if and only
  if for every $s\in \So$, the fiber map $\delta_{\Xi,\mcF,s}$ is
  surjective.
\end{proof}

\mni
The locus of maps from a family of smooth curves has an
$\text{ev}$-obstructed locus that can be detected after base change,
e.g., to a relative Hom space.

\begin{defn}\cite[Section 10]{Douady} \label{defn-H} \marpar{defn-H}
  For $\pi$ as above, for every smooth map of complex analytic spaces,
  $$
  \rho:\mcY\to \So,
  $$
  the \textbf{relative Hom space}, 
  $$
  H_{\pi,\rho} = \text{Hom}_{\So}(\Cc,\mcY) \to \So,
  $$
  is the maximal open subset of the $\So$-relative Douady space of the
  fiber product $\Cc\times_{\So}\mcY$ parameterizing holomorphic, flat
  families of closed analytic subspaces of $\Cc\times_{\So}\mcY$ that
  are proper over the base and such that the first projection map to
  the pullback of $\Cc$ is an isomorphism.
\end{defn}

\mni
Inverting this isomorphism and composing with the second projection
map gives a universal $\So$-morphism,
$$
u:H_{\pi,\rho}\times_{\So} \Cc \to \mcY.
$$
With respect to this relative Hom space, the $\text{ev}$-obstructed
locus can be formulated in terms of the associated extension sheaf.

\begin{defn} \label{defn-Qtoo} \marpar{defn-Qtoo}
  The \textbf{associated extension sheaf} is the coherent analytic
  sheaf on $H_{\pi,\rho}\times_{\So}\Cc$,
  $$
  E_{\pi,\rho}:=\textit{Hom}_{\OO}(u^*\Omega_{\rho},\text{pr}_2^*E_\pi).
  $$ 
  Relative to the first projection morphism,
  $$
  \wt{\pi}:H_{\pi,\rho}\times_{\So} \Cc \to H_{\pi,\rho},
  $$
  the \textbf{obstruction sheaf} is the coherent analytic sheaf on
  $H_{\pi,\rho}$,
  $$
  O_{\pi,\rho} := R^1\wt{\pi}_* E_{\pi,\rho}.
  $$
  The \textbf{non-free locus} $H_{\pi,\rho}^{\text{nf}}$ is the
  support of $O_{\pi,\rho}$, and the \textbf{free locus}
  $H_{\pi,\rho}^{\text{f}}$ is the open complement in $H_{\pi,\rho}$
  of this closed analytic subspace.  For every morphism of complex
  analytic spaces,
  $$
  \zeta:T\to H_{\pi,\rho},
  $$
  the \textbf{pullback non-free locus} $T^{\text{nf}},$ resp, the
  \textbf{pullback free locus} $T^{\text{f}},$ is the inverse image
  under $\zeta$ of $H_{\pi,\rho}^{\text{nf}}$, resp. of
  $H_{\pi,\rho}^{\text{f}}.$
\end{defn}

\begin{lem} \label{lem-free} \marpar{lem-free}
  If $\pi$ is smooth, then the open subset
  $H^{\text{f}}_{\pi,\rho}\times_{\So} \Cc$ of
  $H_{\pi,\rho}\times_{\So}\Cc$ is the smooth locus of $u$.  The
  restriction of $u$ to every analytic subspace of
  $H_{\pi,\rho}^{\text{nf}}\times_{\So}\Cc$ is non-submersive,
  compatibly with arbitrary base-change of $\So$.
\end{lem}  

\begin{proof}
  Denote the dual $\text{Hom}_{\OO_{\mcY}}(\Omega_\rho,\OO_{\mcY})$ by
  $T_{\rho}$.  By Lemma \ref{lem-Q}, the open $H_{\pi,\rho}$ is
  contained in the open complement $H_{\pi,\rho}^o$ of the support of
  $$
  R^1\wt{\pi}_*u^*T_{\rho}.
  $$
  The open $H_{\pi,\rho}^o$ is the maximal open that is smooth over
  $\So$ of relative dimension equal to the expected dimension, and the
  $\So$-relative tangent sheaf on $H_{\pi,\rho}^o$ equals the
  restriction of
  $$
  \wt{\pi}_*\textit{Hom}_{\OO}(u^*\Omega_{\rho},\OO).
  $$
  Moreover, $H_{\pi,\rho}^{\text{f}}$ is the maximal open in
  $H_{\pi,\rho}^o$ such that for every point $s\in \So$ and for every
  point $[v:\Cc_s\to \mcY_s]$ in the fiber $H_{\pi,\rho,s}$, the
  induced map of Zariski tangent spaces,
  $$
  H^0(\Cc_s,v^*T_{\mcY_s}) \to H^0(\Cc_s,v^*\Omega_{\mcY_s})^\vee
  $$
  is surjective.

\mni
The fiber $\Cc_s$ is isomorphic to $\CC\PP^1$.  Thus the locally free
sheaf $v^*\Omega_{\mcY_s}$ is isomorphic to a direct sum of invertible
sheaves $\OO(-a_j)$ for integers $a_j\in \ZZ$.  Of course the induced
map,
$$
H^0(\CC\PP^1,\OO(a_j))\to H^0(\CC\PP^1,\OO(-a_j))^\vee,
$$
is surjective if and only if $a_j\geq 0$.  Thus,
$\pi^{-1}H_{\pi,\rho}^{\text{f}}$ is the unique open subscheme of
$H_{\pi,\rho}^o$ whose points $(s,[v])$ are precisely those maps such
that $v^*T_{\mcY/\So}$ is globally generated.

\mni
Let $(s,[v])$ be a point in $H_{\pi,\rho}^{\text{nf}}$.  Since
constant morphisms have $v^*T_{\mcY/\So}$ isomorphic to a direct sum
of copies of the structure sheaf $\OO_{\CC\PP^1}$, which is globally
generated, the morphism $v$ is not a constant morphism.  Thus, the
derivative map,
$$
dv:T_{\CC\PP^1} \to v^*T_{\mcY/\So},
$$
is an injective homomorphism of coherent analytic sheaves on
$\CC\PP^1$.  Since $T_{\CC\PP^1}$ is globally generated, the
derivative map of $v$ on the factor $T_{\CC\PP^1}$ factors through the
adjointness homomorphism $\alpha$,
$$
\alpha:H^0(\CC\PP^1,v^*T_{\CC\PP^1})\otimes_{\CC} \OO_{\CC\PP^1} \to
v^*T_{\CC\PP^1}.  
$$
Since $(s,[v])$ is not in $H_{\pi,\rho}^{\text{f}}$, the cokernel of
$\alpha$ is a locally free sheaf of positive rank.  Thus, for every
$q\in \CC\PP^1$, the image of the derivative map of $u$ at $(s,[v],q)$
is not surjective.  Thus, the restriction of $u$ to each complex
analytic subvariety of $H_{\pi,\rho}\times_\So\Cc$ containing
$(s,[v],q)$ is not submersive.
\end{proof}
  
\mni
For every $\text{ev}$-dominant class $\eb$, denote by $\mm'_\eb$ the
maximal open subset (possibly empty) over which the following
evaluation map is smooth, i.e., submersive,
$$
\text{ev}_{0,1,\eb}:\Kgnb{0,1}((\mm,\jj),\eb)\to \mm.
$$
Denote by $\mm_\eb$ the open subset of $\mm'_\eb$ over which the
evaluation map is smooth, and every fiber of the evaluation map
parameterizes only maps with smooth domain.

\begin{prop} \label{prop-free} \marpar{prop-free}
  For every connected, compact K\"{a}hler manifold $(\mm,\jj,\oom)$,
  for every $\jj$-irreducible, $\text{ev}$-dominant class $\eb$, the
  complements of both $\mm'_\eb$ and $\mm_\eb$ are proper, closed,
  analytic subvarieties of $\mm$.
\end{prop}

\begin{proof}
  The critical locus of $\text{ev}_{0,1,\eb}$ in
  $\Kgnb{0,1}((\mm,\jj),\eb)$ is a closed analytic subspace.  Since
  the evaluation morphism is proper and holomorphic, the image of this
  closed analytic subspace is a closed analytic subspace of $\mm$.
  The goal is to prove that this closed analytic subspace is proper,
  i.e., it does not equal all of $\mm$.  By definition, $\mm'_\eb$ is
  the open complement of this closed analytic subspace.

\mni
By Proposition \ref{prop-frq}, there exists a nowhere dense, closed
analytic subvariety $Z$ of $\mm$ whose open complement $\mm^o$ is
submersive and relatively projective over $Q^o$, and such that the
open complement $\Kgnb{0,0}((\mm,\jj),\eb)^o$ of $\Kgnb{0,0}(Z,\eb)$
in $\Kgnb{0,0}((\mm,\jj),\eb)$ equals the relative moduli space
$\Kgnb{0,0}(\mm^o/Q^o,\eb)$.  Since $Z$ is nowhere dense, it suffices
to prove that the critical set of $\text{ev}_{0,1,\eb}$ in $\mm^o$ is
nowhere dense.

\mni
Denote by $\Delta^o \subset \Kgnb{0,1}((\mm,\jj),\eb)^o$ the boundary
divisor.  As a closed analytic subspace of
$\Kgnb{0,1}((\mm,\jj),\eb)^o$, this is proper over $Q^o$.  Hence it is
a union of finitely many compact complex analytic spaces each
bimeromorphic to a connected K\"{a}hler manifold that is proper over
$Q^o$.  Each of these finitely many ``irreducible components'' is in
the image $\Delta_{\eb',\eb''}$ of the natural $1$-morphism,
$$
\Kgnb{0,2}((\mm,\jj),\eb')\times_{\text{ev}_2,\mm,\text{ev}}
\Kgnb{0,1}((\mm,\jj),\eb'') \to
\Kgnb{0,1}((\mm,\jj),\eb).
$$
for an ordered pair $(\eb',\eb'')$ of nonzero homology classes with
$\eb'+\eb''=\eb$.  Thus, there are finitely many such pairs
$(\eb',\eb'')$ such that the images $\Delta_{\eb',\eb''}$ cover all of
$\Delta^o$.  The relative complement of $\mm_\eb$ in $\mm'_\eb$ is the
intersection of $\mm'_\eb$ with the union over these finitely many
pairs $(\eb',\eb'')$ of the image $Z_{\eb'}$ of the proper holomorphic
map,
$$
\Kgnb{0,2}((\mm,\jj),\eb')\times_{\text{ev}_2,\mm,\text{ev}}
\Kgnb{0,1}((\mm,\jj),\eb'')
\xrightarrow{\text{pr}_1}
\Kgnb{0,2}((\mm,\jj),\eb')
\xrightarrow{\text{ev}_1}
\mm.
$$
Thus, it suffices to prove that $\mm_\eb$ is nonempty.

\mni
Since $\eb$ is a $\jj$-irreducible $\text{ev}$-dominant class, neither
of $\eb'$ nor $\eb''$ is $\text{ev}$-dominant.  Thus $Z_{\eb'}$ is a
proper closed analytic subspace of $\mm$.  Similarly, the image
$Z_{\eb',\eb''}$ of the following evaluation map is a proper, closed
analytic subspace of $\mm$,
$$
\text{pr}_1\circ \text{ev}_{0,2,\eb'}\circ \text{pr}_1:
\Kgnb{0,2}((\mm,\jj),\eb')\times_{\text{ev}_2,\mm,\text{ev}}
\Kgnb{0,1}((\mm,\jj),\eb'') \to
\mm.
$$
The union in $\mm$ of $Z$ and the proper, closed analytic subspaces
$Z_{\eb'}$ and $Z_{\eb',\eb''}$ for the finitely many pairs
$(\eb',\eb'')$ is a proper, closed analytic subspace $Z'$ of $\mm$,
and $\mm_\eb$ equals $\mm'_\eb\setminus Z'$.  Over the dense, Zariski
open complement $U$ of $Z'$, the evaluation morphism
$$
\text{ev}=\text{ev}_{0,1,\eb}:\Kgnb{0,1}((\mm,\jj),\eb)\to \mm,
$$
parameterizes only stable maps with irreducible domain.

\mni
For the universal family of genus-$0$ stable maps of class $\eb$,
$$
\pi:\Kgnb{0,1}((\mm,\jj),\eb) \to \Kgnb{0,0}((\mm,\jj),\eb), \ \
\text{ev}:\Kgnb{0,1}((\mm,\jj),\eb) \to \mm,
$$
for the holomorphic submersion that is the projection morphism,
$$
\rho:\mm\times \Kgnb{0,0}((\mm,\jj),\eb) \to \Kgnb{0,0}((\mm,\jj),\eb),
$$
the universal map defines a section of the relative Douady space,
$$
\zeta:\Kgnb{0,0}((\mm,\jj),\eb) \to H_{\pi,\rho}.
$$
Denote by $\Kgnb{0,0}((\mm,\jj),\eb)^{\text{nf}}$ the pullback under
$\zeta$ of the non-free locus $H_{\pi,\rho}^{\text{nf}}$.  This is a
closed, analytic subvariety of $\Kgnb{0,0}((\mm,\jj),\eb)$.  Thus, the
inverse image under $\pi$,
$$
\Kgnb{0,1}((\mm,\jj),\eb)^{\text{nf}} \subset \Kgnb{0,1}((\mm,\jj),\eb),
$$
is a closed analytic subvariety.  Since $\text{ev}$ is a proper map of
complex analytic spaces, the image of
$\Kgnb{0,1}((\mm,\jj),\eb)^{\text{nf}}$ in $\mm$ is a closed analytic
subvariety $W$ of $\mm$.  By Lemma \ref{lem-free}, the intersection of
$W$ with the open $U$ is precisely the image in $U$ of the non-smooth
locus of the restricted morphism,
$$
\text{ev}:\text{ev}^{-1}(U)\to U.
$$
This morphism is everywhere non-submersive on the non-smooth locus.
Thus the closed analytic subspace $U\cap W$ of $U$ is proper.  The
dense open complement in $U$ of this closed analytic subspace is
contained in $\mm_\eb$.  Therefore $\mm_\eb$ is nonempty.
\end{proof}

\mni
Ampleness of the psi class in Theorem \ref{thm-simpFano} follows from
the identity of the Gysin classes.  This identity, in turn, follows
from simple facts about the low degree cohomology groups and Picard
group of the total space of a $\mathbb{CP}^1$-fibration.

\begin{defn} \label{def-nsb} \marpar{def-nsb}
  For every proper, holomorphic submersion, $\pi:\Cc\to \So$, of
  connected complex analytic spaces whose fibers are connected,
  genus-$0$ curves, for every section, $\sigma:\So \to \Cc$, of $\pi$,
  the \textbf{normalized section bundle} is
  $\pi^*\psi(\underline{\sigma})$, where $\psi$ is
  $\sigma^*\omega_\pi$.
\end{defn}

\mni
The pullback by $\sigma$ of the normalized section bundle is
isomorphic to $\OO_{\So}$.  Thus, there is a complex,
$$
\ZZ[\pi^*\psi(\underline{\sigma})] \to \text{Pic}(\Cc)
\xrightarrow{\sigma^*} \text{Pic}(\So).
$$
Similarly, via the first Chern class, there is a complex of singular
cohomology groups,
$$
\ZZ\cdot c_1(\pi^*\psi(\underline{\sigma})) \to H^2(\Cc;\ZZ)
\xrightarrow{\sigma^*} H^2(\So;\ZZ).
$$

\begin{lem} \label{lem-psi} \marpar{lem-psi}
  Both of the complexes above are short exact sequences that are split
  by the pullback map $\pi^*$.
\end{lem}

\begin{proof}
  Pullback by the inclusion of the fiber establishes left exactness.
  Since $\pi\circ \sigma$ is the identity, the composition
  $\sigma^*\circ \pi^*$ is the identity.  This establishes right
  exactness as well as the splitting property.  It remains to prove
  exactness in the middle.

\mni
Up to twisting by an appropriate integer power of
$\pi^*\psi(\underline{\sigma})$, every holomorphic invertible sheaf
$\mcL$ on $\Cc$ has relative degree $0$ over $\So$.  By cohomology and
base change, and by vanishing of $h^1(\CC\PP^1,\OO_{\CC\PP^1})$, the
higher direct images $R^q\pi_*\mcL$ are zero for $q>0$, and
$\pi_*\mcL$ is an invertible sheaf $\mcA$.  Moreover, by adjointness
of $\pi_*$ and $\pi^*$, there is a natural morphism of coherent
sheaves,
$$
\pi^*\mcA = \pi^*\pi_*\mcL \to \mcL,
$$
compatible with arbitrary base change of $\So$.  In particular,
restricting to fibers of $\pi$, this morphism is an isomorphism.
Therefore, the invertible sheaf $\mcL$ is trivial if and only if
$\mcA=s^*\mcL$ is also trivial.  That establishes exactness in the
middle for the sequence of Picard groups.  Exactness in the middle for
cohomology follows from the Leray-Serre spectral sequence: existence
of the adjusted section class proves vanishing of the transgression
map to $H^3(\So;\ZZ)$.
\end{proof}

\begin{cor} \label{cor-psi1} \marpar{cor-psi1}
  For every integer $r\geq 1$, with the same hypotheses as above,
  $$
  c_1(\pi^*\psi(\underline{\sigma}))^r =\sigma_*(c_1(\psi)^{r-1}) +
  \pi^* c_1(\psi)^r.
  $$  
  For cohomology classes $D_i\in H^2(\Cc;\ZZ), i=0,\dots,r,$ in the
  kernel of $\sigma^*$ having $\pi$-relative degrees
  $\langle D_i,\eb \rangle \in H^2(\CC\PP^1;\ZZ) = \ZZ$, the Gysin
  pushforward equals
  $$
  \pi_*(D_0\smile \dots \smile D_r) = \langle D_0, \eb \rangle
  \cdots \langle D_r, \eb \rangle c_1(\psi)^r.
  $$  
  By multilinearity, the same holds for $D_i\in H^2(\Cc;\QQ)$,
  resp. $D_i\in H^2(\Cc;\RR)$.
\end{cor}

\begin{proof}
  The first formula is proved by induction on $r$.  For $r=1$, this is
  the Whitney sum formula for $c_1$.  By way of induction, let
  $r\geq 1$ be an integer such that the formula holds.  By the
  Projection Formula, $\sigma^*\sigma_*(\alpha)$ equals
  $-c_1(\psi)\smile \alpha$.  Thus,
  $\sigma^*c_1(\pi^*\psi(\underline{\sigma}))^r$ equals
  $$
  \sigma^*\sigma_*(c_1(\psi)^{r-1}) + \sigma^*\pi^*c_1(\psi)^r =
  -c_1(\psi)\smile c_1(\psi)^{r-1} + c_1(\psi)^r = 0.
  $$  
  Therefore,
  $$
  c_1(\pi^*\psi(\underline{\sigma}))^{r+1} = (\sigma_*(1) +
  \pi^*c_1(\psi)) \smile c_1(\pi^*\psi(\underline{\sigma}))^{r+1} =
  $$
  $$
  \sigma_*(\sigma^*c_1(\pi^*\psi(\underline{\sigma}))^r) +
  \pi^*c_1(\psi)\smile (\sigma_*(c_1(\psi)^{r-1}) + \pi^*c_1(\psi)^r) =
  $$
  $$
  \sigma_*0 + \sigma_*(c_1(\psi)^r) + \pi^*c_1(\psi)^{r+1}.
  $$
  So the formula is proved by induction on $r$.

\mni
Since every $D_i$ is in the kernel of $\sigma^*$, by Lemma
\ref{lem-psi}, $D_i$ equals
$\langle D_i, \eb \rangle c_1(\pi^*\psi(\underline{\sigma}))$.  Thus
$D_0\smile \dots \smile D_r$ equals
$\langle D_0,\eb \rangle \cdots \langle D_r,\eb \rangle
c_1(\pi^*\psi(\underline{\sigma}))^{r+1}$.  Combined with the formula
from the last paragraph and the Projection Formula, this gives the
formula for the Gysin pushforward class.
\end{proof}

\begin{prop} \label{prop-pushforward} \marpar{prop-pushforward}
  For every integer $r\geq -1$, with the same hypotheses as above, the
  $q^{\text{th}}$ higher direct image under $\pi$ of the $r$-times
  twisted invertible sheaf
  $\mc{F}_r := [\pi^*\psi(\ul{\sigma})]^{\otimes r}$ vanishes for
  $q> 0$.  Also, the pushforward $\mc{E}_{r}:= \pi_*\mc{F}_r$ is
  locally free of rank $r+1$ and compatible with arbitrary base change
  of $S$.  For every integer $r\geq 0$, there is a short exact
  sequence,
  $$
  \begin{CD}
  0 @>>> \psi\otimes_{\OO_{\So}} \mc{E}_{r-1} \to \mc{E}_r \to \OO_{\So} \to 0,
  \end{CD}
  $$
  where the second homomorphism is evaluation along $\sigma$.
\end{prop}

\begin{proof}
  All parts are proved by induction on $r$.  When $r$ equals $-1$,
  since all cohomology on fibers is zero, the result follows by
  cohomology and base change.  For every $r\geq 0$, restriction to
  $\sigma$ defines a short exact sequence,
  $$
  \begin{CD}
  0 @>>> \pi^*(\psi)\otimes_{\OO_{\Cc}}\mc{F}_{r-1} @>>> \mc{F}_r @>>>
  \sigma_*\OO_{\So} @>>> 0.
  \end{CD}
  $$
  By the long exact sequence of higher direct images,
  $R^q\pi_*\mc{F}_r$ vanishes if both $R^q\pi_*\sigma_*\OO_{\So}$
  vanishes and $R^q\pi_*\mc{F}_{r-1}$ vanishes.  Since
  $\sigma\circ \pi$ is the identity, of course
  $R^q\pi_*\sigma_*\OO_{\So}$ vanishes for all $q>0$.  By the
  induction hypothesis, also $R^q\pi_*\mc{F}_{r-1}$ vanishes for all
  $q>0$.  Thus, also $R^1\pi_*\mc{F}_r$ vanishes for all $q>0$.  Thus,
  the long exact sequence reduces to a short exact sequence,
  $$
  \begin{CD}
  0 @>>> \psi\otimes_{\OO_{\So}} \mc{E}_{r-1} \to \mc{E}_r \to \OO_{\So} \to 0.
  \end{CD}
  $$
  By the induction hypothesis, the first term is locally free of rank
  $r$, so that also the third term is locally free of rank $r+1$.
\end{proof}

\begin{notat} \label{notat-free} \marpar{notat-free}
  For every connected, compact K\"{a}hler manifold $(\mm,\jj,\oom)$,
  for every $\jj$-irreducible, $\text{ev}$-dominant class $\eb$, for
  the evaluation morphism
  $$
  \text{ev}:\Kgnb{0,1}((\mm,\jj),\eb) \to \mm,
  $$
  denote the inverse image open subset $\text{ev}^{-1}(\mm_\eb)$ by
  $\Kgnb{0,1}^o((\mm,\jj),\eb)$.  Denote by
  $\pi_{\eb}^o:\Cc_{\eb}^o\to \Kgnb{0,1}^o((\mm,\jj),\eb)$ the
  following fiber product,
  $$
  \begin{CD}
  \Cc_{\eb}^o @>\text{pr}_1 >> \Kgnb{0,1}((\mm,\jj),\eb) \\
  @V \pi_{\eb}^o VV @VV \pi V \\
  \Kgnb{0,1}^o((\mm,\jj),\eb) @>> \pi|_{F} > \Kgnb{0,0}((\mm,\jj),\eb)
  \end{CD}.
  $$
  The graph of the inclusion morphism
  $\Kgnb{0,1}^o((\mm,\jj),\eb) \hookrightarrow
  \Kgnb{0,1}((\mm,\jj),\eb)$ is a section of $\pi_{\eb}^o$, denoted
  $$
  \sigma_{\eb}^o:\Kgnb{0,1}^o((\mm,\jj),\eb) \to \Cc_{\eb}^o.
  $$
  The composition $\text{ev}\circ \text{pr}_1$ is a morphism of complex
  analytic spaces, denoted
  $$
  u_{\eb}^o:\Cc_{\eb}^o \to \mm.
  $$
  For every complex analytic space $\So$ and for every morphism of
  complex analytic spaces $\zeta:\So\to \Kgnb{0,1}^o((\mm,\jj),\eb)$,
  denote the pullback diagram by
  $$
  \zeta: \ \ (\Cc_\zeta \xrightarrow{\pi_\zeta} \So, \sigma_\zeta:\So
  \to \Cc_\zeta, u_\zeta: \Cc_\zeta \to \mm).
  $$
\end{notat}

\begin{prop} \label{prop-NK} \marpar{prop-NK}
  For every connected, compact K\"{a}hler manifold $(\mm,\jj,\oom)$,
  for every $\jj$-irreducible, $\text{ev}$-dominant class $\eb$, for
  every morphism $\zeta:\So \to \Kgnb{0,1}^o((\mm,\jj),\eb)$ of
  complex analytic spaces such that $\text{ev}\circ \zeta$ is constant
  with image $x\in\mm$, the pullback class $u_\zeta^*[\oom]$ equals
  $\langle \oom,\beta \rangle c_1(\pi^*\psi(\underline{\sigma}))$.  If
  there exists a section $\wt{\sigma}:\So \to \Cc_\zeta$ that is
  disjoint from $\sigma$ and with $u_\zeta\circ\wt{\sigma}$ constant,
  then $\zeta$ is locally constant.  Conversely, if $\So$ is compact
  of (pure) complex dimension $r$ and if $\zeta$ is generically finite
  to its (closed analytic) image, then $\int_{\So}\zeta^* c_1(\psi)^r$
  is positive.
\end{prop}

\begin{proof}
  Since $u_\zeta\circ \sigma$ equals $\text{ev}\circ \zeta$, also
  $\sigma^*u_\zeta^*$ is the zero homomorphism.  Thus, by Lemma
  \ref{lem-psi}, $u_\zeta^*[\omega]$ equals a multiple of the first
  Chern class of the adjusted section bundle.  Restricting to a fiber
  of $\pi_\zeta$, that multiple equals $\langle \oom, \beta\rangle$.

\mni
If there exists a section $\wt{\sigma}$ that is disjoint from
$\sigma$, then $\wt{\sigma}^*u_\zeta^*[\omega]$ equals
$\langle \oom,\beta \rangle c_1(\psi)$ in $H^2(\So;\RR)$.  If
$u_\zeta\circ \wt{\sigma}$ is constant, then $c_1(\psi)$ is contained
in the torsion kernel of $H^2(\So;\ZZ) \to H^2(\So;\RR)$.  This
torsion class determines a finite, unbranched cover of $\So$.  Up to
replacing $\So$ by this cover, $c_1(\psi)$ is zero.  By the Lefschetz
$(1,1)$ Theorem, the holomorphic invertible sheaf $\psi$ is trivial.
By Lemma \ref{lem-psi}, the invertible sheaf
$\OO_{\Cc_\zeta}(\underline{\wt{\sigma}})$ equals
$\OO_{\Cc_\zeta}(\underline{\sigma})$.  Thus, the two members $\sigma$
and $\wt{\sigma}$ span a pencil of divisors associated to this common
invertible sheaf, and that pencil defines an isomorphism of
$\Cc_\zeta$ with $\CC\PP^1\times \So$ over $\So$ such that the two
sections $\sigma$ and $\wt{\sigma}$ are the zero and infinity
sections.  Since $u_\zeta^*[\oom]$ is the pullback of a class via the
projection
$$
\text{pr}_{\CC\PP^1}: \Cc_\zeta \to \CC\PP^1,
$$
there exists a unique morphism $u:\CC\PP^1\to \mm$ such that $u_\zeta$
equals $u\circ \text{pr}_{\CC\PP^1}$.  Thus, the morphism $\zeta$ is
constant.

\mni
By Corollary \ref{cor-psi1},
$$
u_\zeta^*[\oom]^r = \langle \oom, \beta
\rangle^r\lt(\sigma_*(\zeta^*c_1(\psi)^{r-1}) + \pi^*\zeta^*c_1(\psi)^r\rt),
$$
where $\sigma_*$ is the Gysin pushforward homomorphism.  Assume now
that $\So$ is compact and connected of dimension $r$ and that $\zeta$
is generically finite to its image.  To prove that the integral is
positive, it suffices to replace $\So$ by the image of $\zeta$, since
the integral on $\So$ equals the generic degree of $\zeta$ times the
integral on the image of $\So$.  Thus, assume that $\So$ is a closed
analytic subspace of $F_{\eb,x}$.  The claim is that the associated
morphism
$$
u_\zeta:\Cc_\zeta\to \mm
$$
has finite fibers except over $x$.  For any
$\wt{x}\in \mm\setminus\{x\}$, for the inverse image closed analytic
subspace $\wt{\So} = u_\zeta^{-1}(\{\wt{x}\})$ of $\Cc_\zeta$, for the
composition
$$
\zeta\circ \pi_\zeta:\wt{\So} \to \So \to F_{\eb,x},
$$
the base change family has a section $\wt{\sigma}$ with
$u_{\wt{\zeta}}\circ \wt{\sigma}$ constant to $\wt{x}$.  By the
previous paragraph, $\wt{\zeta}$ is locally constant, i.e., $\wt{\So}$
is a finite set.  This proves the claim.

\mni
Since $\Cc_\zeta$ is a compact complex analytic space of pure
dimension $r+1$, and since $u_\zeta$ is generically finite, the
integral
$$
\int_{\Cc_\zeta} u_\zeta^*[\oom]^{r+1}
$$
is positive.  Since $\zeta^*c_1(\psi)^{r+1}$ is zero on the
$r$-dimensional complex analytic space $\So$, the inductive formula
above gives
$$
\int_{\Cc_\zeta} u_\zeta^*[\oom]^{r+1} = \langle \oom,\eb
\rangle^{r+1}
\int_{\Cc_\zeta} s_* \zeta^*c_1(\psi)^r = \langle \oom,\eb
\rangle^{r+1} \int_{\So} \zeta^* c_1(\psi)^r.
$$
Thus, the integral of $\zeta^*c_1(\psi)^r$ is positive.
\end{proof}

\begin{cor} \label{cor-NK} \marpar{cor-NK}
  For every connected, compact K\"{a}hler manifold $(\mm,\jj,\oom)$,
  for every $\jj$-irreducible, $\text{ev}$-dominant class $\eb$, for
  the Barlet space $\text{Barlet}(\mm,\jj)$, and for the family $\xi$
  of $1$-cycles in $(\mm,\jj)$ that is the image of
  $$
  (\pi_{\eb}^o,u_{\eb}^o):\Cc_{\eb}^o \to \Kgnb{0,1}^o((\mm,\jj),\eb)
  \times \mm,
  $$
  the associated morphism of complex analytic spaces,
  $$
  (\text{ev},\xi):\Kgnb{0,1}^o((\mm,\jj),\eb) \to \mm_\eb \times
  \text{Barlet}(\mm,\jj), 
  $$ 
  is finite.  In particular, for every $x\in \mm_\eb$, the fiber
  $F_{\eb,x}$ of $\text{ev}$ is in Fujiki class $C$.
\end{cor}

\begin{proof}
  By Proposition \ref{prop-NK}, for every $x\in \mm_\eb$, for every
  genus-$0$, $1$-pointed stable map,
  $$
  (\Cc,p,v:\Cc\to \mm_\eb)
  $$ 
  parameterized by a point of $F_{\eb,x}$, for every
  $\wt{x}\in v(\Cc)\setminus \{x\}$ there are at most finitely many
  other points of $F_{\eb,x}$ parameterizing maps with image
  containing $\wt{x}$. In particular, there are at most finitely many
  points whose image cycle coincides with $v_*[C]$.  Thus the morphism
  $(\text{ev},\xi)$ has finite fibers.  Since $\text{ev}$ is a proper
  morphism of complex analytic spaces, the morphism $(\text{ev},\xi)$
  is a finite morphism of complex analytic spaces.  By \cite{Fujiki},
  every irreducible component of the Barlet space of the K\"{a}hler
  manifold $(\mm,\jj)$ is in Fujiki class $C$.  Since $F_{\eb,x}$
  admits a finite morphism of complex analytic spaces to a compact
  complex analytic space in Fujiki class $C$, also $F_{\eb,x}$ is in
  Fujiki class $C$.
\end{proof}

\begin{defn} \label{defn-ah} \marpar{defn-ah}
  For each normalization $Q\to \text{Barlet}(\mm,\jj)$ of an
  irreducible, reduced closed analytic subspace of the Barlet space,
  the \textbf{associated cycle} $Z\subset Q\times \mm$ is the pullback
  to $Q$ of the universal cycle.  The \textbf{isomorphic open}
  $\mm_Q\subset \mm$ is the maximal open subscheme of $\mm$ (possibly
  empty) over which the projection $\text{pr}_{\mm}:Z\to \mm$ is an
  isomorphism.  The \textbf{fundamental locus} $\text{Fund}(Q)$ is the
  image in $Q$ of the closed analytic subspace
  $Z\setminus \text{pr}_{\mm}^{-1}(\mm_Q)$ of $Z$, and $Q^o$ is the
  open complement.
\end{defn}

\begin{rmk} \label{rmk-ah} \marpar{rmk-ah}
  The family of cycles is an almost holomorphic fibration in the sense
  of Definition \ref{defn-frq} if and only if $\mm_Q$ and $Q^o$ are
  nonempty.
\end{rmk}

\mni
For an almost holomorphic fibration, the complement $Q^o$ of the
fundamental locus is a dense open of $Q$, there is a unique open
$\mm_Q^o\subset \mm_Q$ such that $\text{pr}_Q^{-1}(Q^o)$ equals
$\text{pr}_{\mm}^{-1}(\mm_Q^o)$, and the composition
$\phi=\text{pr}_Q\circ \text{pr}_{\mm}^{-1}$,
$$
\phi : \mm_Q^o \to \text{pr}_{\mm}^{-1}(\mm_Q^o)\to Q^o
$$
is a surjective, proper morphism of irreducible, normal complex
analytic spaces of Fujiki class C whose fibers are all
pure-dimensional.  By Sard's Theorem, the image in $Q^o$ of the
non-smooth locus of $\phi$ is a proper closed analytic subspace whose
open complement $Q^{\text{sm}}$ is smooth, and such that the
restriction of $\phi$ to the inverse image
$\mm_Q^{\text{sm}} = \phi^{-1}(Q^{\text{sm}})$,
$$
\phi^{\text{sm}}: \mm_Q^{\text{sm}} \to Q^{\text{sm}},
$$
is a surjective, proper, smooth morphism, i.e., a ``holomorphic
fibration'' (in the sense of Ehresmann's theorem).  Thus, $\phi$ is an
``almost holomorphic fibration''.

\mni
For every connected, compact K\"{a}hler manifold $(\mm,\jj,\oom)$, for
every $\jj$-irreducible, $\text{ev}$-dominant class $\eb$, recall that
$\mm_\eb\subset \mm$ is the maximal open over which
$\text{ev}_{0,1,\eb}$ is smooth and disjoint from the boundary
$\Delta$.  By Proposition \ref{prop-free}, this is a dense open whose
closed complement is an analytic subvariety of $\mm$.  Also recall
that $\Kgnb{0,1}^o((\mm,\jj),\eb) = \text{ev}^{-1}(\mm_\eb)$ is
defined to be the inverse image of $\mm_\eb$ under
$\text{ev}_{0,1,\eb}$, so that the restriction of the evaluation
morphism,
$$
\text{ev}: \Kgnb{0,1}^o((\mm,\jj),\eb) \to \mm_\eb,
$$
is smooth and parameterizes only free $1$-pointed maps with
irreducible domain.

\begin{prop} \label{prop-psi} \marpar{prop-psi}
  For every connected, compact K\"{a}hler manifold $(\mm,\jj,\oom)$,
  for every $\jj$-irreducible, $\text{ev}$-dominant class $\eb$, the
  restriction of $\psi$ to $\Kgnb{0,1}^o((\mm,\jj),\eb)$ is
  $\text{ev}$-ample.  Equivalently, for every $x\in \mm_\eb$, the
  fiber $F_{\eb,x}$ of $\text{ev}$ over $x$ is a complex projective
  manifold on which the restriction of $\psi$ is ample.
\end{prop}

\begin{proof}
  By openness of ampleness, $\text{ev}$-relative ampleness of the
  restriction of $\psi$ is equivalent to ampleness of $\psi$ on every
  fiber $F_{\eb,x}$ of $\text{ev}$.

\mni
Let $Q\to \text{Barlet}(\mm,\jj)$ be the rational quotient of
$(\mm,\jj)$.  Then the open $\mm_Q^{\text{sm}}$ is dense in $\mm$.  By
Proposition \ref{prop-free}, also $\mm_\eb$ is a dense open in $\mm$.
Thus, the intersection $\mm_\eb\cap \mm_Q^{\text{sm}}$ is a dense open
in $\mm$.

\mni
For every $x\in \mm_Q^{\text{sm}} \cap \mm_\eb$, every map
parameterized by $F_{\eb,x}$ has irreducible domain and $\text{ev}$ is
smooth at this genus-$0$, $1$-pointed map, i.e., the map is a free map
from $\CC\PP^1$.  Moreover, the image contains $x$, a point of
$\mm_Q^o$.  Thus, these free maps all have image in the fiber $\mm_q$
of $\phi$ over $q=\phi(x)$.  So $F_{\eb,x}$ is an open and closed
subspace of the space of stable maps in $\mm_q$.  Also $\mm_q$ is
projective by Proposition \ref{prop-frq}.  Thus, $F_{\eb,x}$ is a
projective manifold as well: the morphism from $F_{\eb,x}$ to the Chow
scheme of $\mm_q$ is finite, and the components of the Chow scheme of
a projective scheme are themselves projective.  For all
$x\in \mm_\eb$, by Corollary \ref{cor-NK}, the fiber $F_{\eb,x}$ is a
compact, complex manifold in Fujiki class C.  Thus, the standard
results of Hodge theory still apply for the proper, holomorphic
submersion,
$$
\text{ev}:\Kgnb{0,1}^o((\mm,\jj),\eb) \to \mm_\eb.
$$
In particular, the Hodge numbers $h^{0,1}$ of the fibers are constant.
By \cite{Popovici}, since the fiber $F_{\eb,x}$ is projective for the
general $x\in \mm_\eb$, the fiber is Moishezon for every
$x\in \mm_\eb$.  Finally, by Proposition \ref{prop-NK}, the
restriction of $\psi$ satisfies the hypotheses of the Nakai-Moishezon
criterion for ampleness of the holomorphic invertible sheaf $\psi$ on
the compact, Moishezon manifold $F_{\eb,x}$.  Therefore $\psi$ is
ample on $F_{\eb,x}$.
\end{proof}

\begin{cor} \label{cor-psi} \marpar{cor-psi}
  With the same hypotheses as in Proposition \ref{prop-psi}, for every
  integer $r\geq 0$, for every $x\in \mm_\eb$, the restriction to
  $F_{\eb,x}$ of the pushforward
  $\mc{E}_r:=\pi_*([\pi^*\psi(\ul{\sigma})]^{\otimes r})$ is a locally
  free sheaf of rank $r+1$ with a natural surjection to
  $\OO_{F_{\eb,x}}$ (evaluation at the marked point). The kernel of
  this surjection is a locally free sheaf of rank $r$ that is ample
  when $r\geq 1$.  The restriction to $F_{\eb,x}$ of the twist
  $\psi\otimes \mc{E}_r$ is an ample locally free sheaf of rank $r+1$.
\end{cor}

\begin{proof}
  By Proposition \ref{prop-pushforward}, the pushforward is a locally
  free sheaf of rank $r+1$ with a natural surjection to the structure
  sheaf.  Ampleness of the kernel and ampleness of the twist by $\psi$
  are proved by induction on $r$.  When $r=0$, then $\mc{E}_0$ equals
  the structure sheaf, so that the twist by $\psi$ equals $\psi$.
  This is ample by the previous result.  Thus, assume that $r\geq 1$,
  and assume that the result is true for $r-1$.

\mni
The kernel of the natural surjection equals
$\psi \otimes \mc{E}_{r-1}$, and this is ample by the induction
hypothesis.  Since $\mc{E}_r$ is an extension of the structure sheaf
by an ample locally free sheaf, also $\psi\otimes \mc{E}_r$ is an
extension of $\psi$ by the twist of an ample locally free sheaf by
$\psi$.  By the previous result, $\psi$ is ample.  The tensor product
of an ample locally free sheaf by an ample invertible sheaf is ample.
Finally, an extension of an ample sheaf by an ample sheaf is ample.
Altogether, $\psi\otimes \mc{E}_r$ is ample.  Thus the result holds by
induction on $r$.
\end{proof}

\begin{proof}[Proof of Theorem \ref{thm-simpFano}]
  By definition,
  $$
  f_\eb = \langle \tau_{m_\eb}(\eta_X) \rangle^{\mm,\oom}_{0,1,\eb} =
  \text{GW}^{\mm,\oom}_{0,1,\eb}(q^*c_1(\psi_1)^{m_\eb} \smile
  \text{pr}_1^*(\eta_\mm)) =
  \text{GW}^{\mm,\oom}_{0,1,0,\eb}(q^*c_1(\psi_1)^{m_\eb}).
  $$
  Thus, if $\eb$ is symplectically free, then also $\eb$ is
  symplectically pseudo-free.  Similarly, if $\eb$ is symplectically
  free, then the pairing of $\text{ev}_{0,1,\eb}^*(\eta_\mm)$ with the
  virtual fundamental class is nonzero, so that $\text{ev}_{0,1,\eb}$
  is surjective (else choose a representative for the Poincar\'{e}
  dual homology class of $\eta_\mm$ that is contained in the
  complement of the image).  By Sard's Theorem, for a proper,
  holomorphic map between compact complex analytic varieties, the
  image in the target of the submersive locus in the domain is an open
  (possibly empty) that is the complement of a closed analytic
  subvariety of the target.  If $\eb$ is free, then the submersive
  locus of $\text{ev}_{0,1,\eb}$ is nonempty, so that the image of the
  submersive locus is a dense open subset of $\mm$.  Since
  $\text{ev}_{0,1,\eb}$ is proper, the image of $\text{ev}_{0,1,\eb}$
  is closed, and it contains this dense open.  Therefore
  $\text{ev}_{0,1,\eb}$ is surjective, i.e., $\eb$ is
  $\text{ev}$-dominant.
 
\mni
Next, assume that $\eb$ is $\text{ev}$-dominant and $\jj$-irreducible.
By Proposition \ref{prop-free}, $\eb$ is free.  The goal is to prove
that $\eb$ is symplectically free (closing the circle for all notions
of ``free'' for $\jj$-irreducible classes).  Also, the goal is to
prove the formula for the Gysin class
$\pi_*u^*(D_0\smile \dots \smile D_r)$ in $H^{2r}(F_{\eb,x};\RR)$.  By
Proposition \ref{prop-free}, the open subspace $\mm_\eb$ is a dense
open whose complement is a complex analytic subspace.  In particular,
this open is nonempty.  By definition, for every $x\in \mm_\eb$, the
fiber $F_{\eb,x}$ is a compact complex manifold of dimension $m_\eb$.
By Proposition \ref{prop-psi}, the restriction of $\psi$ on this fiber
is ample.  Also, for $D_i\in H^2(\mm;\RR), i=0,\dots,r$, the pullback
of each $D_i$ by $u_\eb^o$ is a class whose further pullback by
$\sigma_\eb^o$ vanishes when restricted to $F_{\eb,x}$, since this is
also the pullback by the (constant) evaluation morphism.  Thus,
Corollary \ref{cor-psi1} gives the formula for the Gysin pushforward
of the cup product of the pullbacks.  It remains only to prove that
$\eb$ is a symplectically free class.

\mni
Since the restriction of $\psi$ is ample, the following integral is
positive,
$$
\int_{F_{\eb,x}} c_1(\psi)^{m_\eb}.
$$
By the projection formula, this equals the cap product of
$c_1(\psi)^{m_\eb}$ with the pushforward to
$H^*(\Kgnb{0,1}((\mm,\jj),\eb);\QQ)$ of the fundamental class of
$F_{\eb,x}$.  Since $F_{\eb,x}$ is a regular fiber of the evaluation
morphism,
$$
\text{ev}:\Kgnb{0,1}((\mm,\jj),\eb) \to \mm,
$$
the pushforward of the fundamental class equals the cap product of
$\text{ev}^*\eta_\mm$ with the virtual fundamental class of
$\Kgnb{0,1}((\mm,\jj),\eb)$.  Thus, the positive integral equals the
gravitational descendant,
$$
\langle \tau_{m_\eb}(\eta_\mm) \rangle^{\mm,\oom}_{0,\eb} =
\int_{[\Kgnb{0,1}((\mm,\jj),\oom)]^{\text{vir}}}
c_1(\psi)^{m_\eb}\smile \text{ev}^*\eta_\mm.
$$ 
Therefore $\eb$ is a symplectically free class.  

\mni
The next goal is to prove that the set of $\oom$-degrees of
$\text{ev}$-dominant classes is discrete.  Thus, there exists an
$\oom$-minimal, hence $\jj$-irreducible, element of $\snef{\oom}(\mm)$
if and only if $\snef{\oom}(\mm)$ is nonzero.  Thus, assume that there
exists an $\text{ev}$-dominant class.  Then the fibers $\mm_q$ of the
rational quotient are nontrivial.  For a very general fiber $\mm_q$,
every genus-$0$ stable map whose image intersects $\mm_q$ factors
through $\mm_q$.  In particular, every $\text{ev}$-dominant class in
$\mm$ is the pushforward from $\mm_q$ of an $\text{ev}$-dominant class
in $\mm_q$.  Thus, to prove discreteness of the $\oom$-degrees of
$\text{ev}$-dominant class in $\mm$, it is equivalent to prove
discreteness of the $\oom$-degrees of $\text{ev}$-dominant classes in
$\mm_q$.

\mni
Every general fiber $\mm_q$ is a connected projective manifold that is
rationally connected.  Thus the restriction of $\oom$ is $\RR$-ample,
i.e., it equals a linear combination with positive real coefficients
of ample integer divisors, say
$$
r_1 [D_1] + \dots + r_\ell [D_\ell], \ \ 0 < r_1 \leq \dots \leq r_\ell.
$$
For every real number $R>0$, for the integer part $N=\lfloor R/r_1
\rfloor$, there are at most $N^\ell/\ell!$ ordered $\ell$-tuples of
non-negative integers $(n_1,\dots,n_\ell)$ with $r_1n_1+\dots +r_\ell
n_\ell$ bounded by $R$.  Thus, the set of effective curve classes in
$\mm_q$ with $\oom$-degree bounded by $R$ is finite.  Thus, the set of
$\oom$-degrees of $\text{ev}$-dominant curve classes is discrete.

\mni
By discreteness of the $\oom$-degrees of $\text{ev}$-dominant classes,
for every $\text{ev}$-dominant class $\eb$, there are at most finitely
many $\text{ev}$-dominant classes $\eb'$ with $\eb$ equal to
$\eb'+\eb''$ for a class $\eb''$ that is a (nonzero) sum of classes of
genus-$0$ maps (not necessarily $\text{ev}$-dominant).  In particular,
one of these finitely many classes $\eb'$ has minimal $\oom$-degree,
and thus that class is a $\jj$-irreducible, $\text{ev}$-dominant
class, which is then symplectically free.  Thus, every
$\text{ev}$-dominant class $\eb$ is the sum of a $\jj$-irreducible,
symplectically free class $\eb'$ and a sum (possibly zero) of classes
of genus-$0$ maps (not necessarily $\text{ev}$-dominant).  In
particular, if $\snef{\oom}(\mm)$ is nonzero, then there exists a
$\jj$-irreducible, symplectically free class $\eb$.

\mni
Finally, for every $\jj$-irreducible, symplectically free $\eb$, for
every $q\in \mm_\eb$, by Proposition \ref{prop-psi}, the restriction
of $\psi$ to the fiber $F=\text{ev}^{-1}(\{q\})$ is ample.  Also $F$
has pure dimension $m_\eb$.  Assume that $m_\eb>0$.  Then by
\cite{dJS3}, the $\psi$-degree of the first Chern class of $F$ equals
the second gravitational descendant,
$$
s_{\eb} = s_{\eb}(\mm) := \lt \langle \tau_{m_\eb-1}(\eta_m),
  \text{ch}_2(T^{1,0}_{\mm,\oom}) +
  \frac{m_\eb}{2(m_\eb+2)^2}c_1(T^{1,0}_{\mm,\oom})^2
  \rt\rangle^{\mm,\oom}_{0,\eb}.
$$
By the asymptotic Riemann-Roch formulas of \cite{KollarMatsusaka},
\cite{Luo}, and \cite{Matsusaka}, even if $\psi$ is merely big and
nef, there is an asymptotic formula for $d\gg 0$,
$$
h^0(F,\psi^{\otimes d}) 
= \frac{f_\eb}{m_\eb!}d^{m_\eb} +
  \frac{s_\eb}{2(m_\eb-1)!} d^{m_\eb-1} + \dots =
  \frac{f_\eb d^{m_\eb}}{m_\eb!}\lt( 1 + \frac{m_\eb q_\eb}{2d} +
  \dots \rt).
$$
\end{proof}


\section{Proof of Theorem \ref{thm-surf}} \label{sec-Mori}
\marpar{sec-Mori}

\mni
Let $(\mm,\jj,\oom)$ be a connected, compact K\"{a}hler manifold.
Denote the rational quotient by
$$
Q\to \text{Barlet}(\mm), \ \ \mm\supseteq \mm^o \xrightarrow{\phi}
Q^o\subseteq Q.    
$$

\begin{proof}[Proof of (1) of Theorem \ref{thm-surf}]
  If $\snef{\oom}(\mm)$ is nonzero, then by Theorem
  \ref{thm-simpFano}, the rational quotient $\phi$ has positive fiber
  dimension.

\mni
Since $Q$ is finite over the Barlet space, it is in Fujiki class $C$.
Thus, if $Q$ is not a point, then there exists a nonzero element in
$H^2(Q;\QQ)$.  The pullback in $H^2(\mm;\QQ)$ is orthogonal to every
homology class $\eb$ of a curve contained in a fiber of the rational
quotient.  In particular, it is orthogonal to every connected tree of
rational curves such that at least one component is a general free
curve.  Thus, the class is orthogonal to every symplectically
pseudo-free class.  When these classes span $H_2(\mm;\QQ)$, then the
pullback is zero.  Therefore $Q$ is a point.
\end{proof}

\mni
The remainder of Theorem \ref{thm-surf} follows from the next three
propositions.

\begin{prop} \label{prop-mainA} \marpar{prop-mainA}
  For every $\oom$-minimal symplectically free class $\eb$ that is
  symplectically $2$-free, for every point $x\in \mm_\eb$, at least
  one connected component of $F_{\eb,x}$ is uniruled.  Moreover, each
  surface in $\mm$ swept out by a rational curve $R$ in $F_{\eb,x}$ is
  a rational surface on which the $\eb$-curves parameterized by $R$
  move in the pencil $R$.  Conversely, if $F_{\eb,x}$ is weakly Fano,
  then $\eb$ is symplectically $2$-free.
\end{prop}

\begin{proof}
  By Proposition \ref{prop-free}, the $\eb$-generic locus $\mm_\eb$ is
  a dense open whose complement is a proper closed analytic subspace
  of $\mm$.  For the inverse image open
  $\Kgnb{0,1}^o((\mm,\jj),\eb) := \text{ev}^{-1}(\mm_\eb)$, the
  restricted evaluation map,
  $$
  \text{ev}:\Kgnb{0,1}^o((\mm,\jj),\eb) \to \mm_\eb,
  $$
  is a proper, smooth morphism of complex analytic spaces.  Every
  fiber has pure dimension $m_\eb$, and the restriction of $\psi$ is
  $\text{ev}$-relatively ample.  Thus, the $\text{ev}$-relative spaces
  of genus-$0$, $1$-pointed stable maps to fibers of $\text{ev}$ with
  specified curve class are themselves proper over
  $\Kgnb{0,1}^o((\mm,\jj),\eb)$ (for the evaluation map).  To prove
  uniruledness of at least one connected component of the fiber
  $F_{\eb,x}$ of $\text{ev}$ over \textbf{every} point $x\in \mm_\eb$,
  it suffices to prove uniruledness for general $x\in \mm_\eb$.

\mni
Denote by $Q\to \text{Barlet}(\mm,\jj)$ the rational quotient of
$(\mm,\jj)$.  Denote by $\phi^{\text{sm}}$ the associated
maximally-extended proper, smooth morphism of complex analytic spaces
defined on dense open subspaces of $\mm$, resp, of $Q$,
$$
\lt( \mm \supseteq \rt)
\mm_Q^{\text{sm}} 
\xrightarrow{\phi} Q^{\text{sm}} \lt( \subseteq Q \rt),
$$ 
such that all fibers of $\phi^{\text{sm}}$ are rationally connected
and every free rational curve that intersects $\mm_Q^{\text{sm}}$ is
contained in a fiber of $\phi^{\text{sm}}$.  The restriction of $\oom$
to the dense open $\mm_Q^{\text{sm}}$ is K\"{a}hler.  Since the fibers
of $\phi^{\text{sm}}$ are also rationally connected, compact, complex
manifolds, they are projective, i.e., $\phi^{\text{sm}}$ is (weakly)
projective.

\mni
For each point $q\in Q^{\text{sm}}$, for each point $x$ in the fiber
$\mm_q = \phi^{-1}(q)$, the fiber $F_{\eb,x}$ of $\text{ev}$ equals
the fiber over $x$ of
$$
\text{ev}_q:\Kgnb{0,1}((\mm_q,\jj_q),\eb) \to \mm_q,
$$
since every free curve in $\mm$ that contains $x$ is contained in
$\mm_q$.  Since $\mm_q$ is a fiber of the proper morphism
$\phi^{\text{sm}}$, there is a normal bundle short exact sequence of
holomorphic vector bundles on $\mm_q$,
$$
0 \to T^{1,0}_{\mm_q,\jj_q} \to T^{1,0}_{\mm,\jj}|_{\mm_q} \to
\OO_{\mm_q}^{\text{dim}(Q)} \to 0,
$$
so that $\text{ch}_r$ of the tangent bundle of $\mm_q$ equals the
restriction of $\text{ch}_r$ of the tangent bundle of $\mm$ for every
$r\geq 1$.

\mni
By \cite{dJS3}, the relative first Chern class of $\text{ev}_q$ equals
$$
-\text{ev}_q^*c_1(T^{1,0}_{\mm_q,\jj_q}) -\psi +
\pi_*u^*\lt(\text{ch}_2(T^{1,0}_{\mm_q,\jj_q})  +
\frac{1}{2(m_\eb+2)}c_1(T^{1,0}_{\mm_q,\jj_q})^2 \rt).
$$ 
The first summand vanishes when restricted to a fiber of
$\text{ev}_q$.  By Corollary \ref{cor-psi1}, we have a divisor class
relation,
$$
-\psi + \pi_*u^*\lt( \frac{1}{2(m_\eb+2)} c_1(T^{1,0}_{\mm_q,\jj_q})^2\rt) =
\pi_*u^*\lt(\frac{m_\eb}{2(m_\eb+2)^2} c_1(T^{1,0}_{\mm_q,\jj_q})^2\rt),
$$
i.e.,
$$
\pi_*u^*\lt( c_1(T^{1,0}_{\mm_q,\jj_q})^2\rt) = \langle
c_1(T^{1,0}_{\mm_q,\jj_q}),\eb \rangle^2 \psi.
$$
Thus, for every $x\in \mm_q$, the restriction to the fiber $F_{\eb,x}$
of $\text{ev}_q$ of the relative first Chern class of $\text{ev}_q$
gives the first Chern class of the fiber,
$$
c_1(T^{1,0}_{F_{\eb,x}}) = \pi_*u^*\lt(\text{ch}_2(T^{1,0}_{\mm_q,\jj_q})  + 
\frac{m_\eb}{2(m_\eb+2)^2}c_1(T^{1,0}_{\mm_q,\jj_q})^2 \rt). 
$$
Since $\psi$ is ample on $F_{\eb,x}$, there exists an integer $d> 0$
such that $\psi^{\otimes d}$ is very ample.  Thus, $F_{\eb,x}$ is
covered by smooth, projective curves $A$ that are complete
intersections of $m_\eb-1$ divisors in the linear system of the very
ample invertible sheaf $\psi^{\otimes d}$.  The total degree on $A$ of
the restriction of the first Chern class equals
$$
\langle c_1(T^{1,0}_{F_{\eb,x}}), [A] \rangle = d^{m_\eb-1}
\int_{F_{\eb,x}} c_1(\psi)^{m_\eb-1} \smile c_1(T^{1,0}_{F_{\eb,x}}) =
$$
$$
d^{m_\eb-1}\lt\langle \tau_{m_\eb-1}(\eta_\mm), \text{ch}_2(T^{1,0}_{\mm_q,\jj_q})  + 
\frac{m_\eb}{2(m_\eb+2)^2}c_1(T^{1,0}_{\mm_q,\jj_q})^2 \rt\rangle^{\mm,\oom}_{0,\eb}.
$$
In particular, the total degree on $A$ of the first Chern class is
positive if and only if $s_\eb$ is positive.

\mni
If the fiber $F_{\eb,x}$ is ``Fano in the weak sense'' that the first
Chern class equals a nonzero, pseudo-effective divisor class, then the
intersection number of that divisor class with the moving curve class
$[A]$ is strictly positive.  Then $s_\eb$ is also positive.

\mni
Conversely, if $s_\eb$ is positive, then the first Chern class of
$F_{\eb,x}$ has positive degree on at least one connected component of
$A$.  By Bertini's Connectedness Theorem, each connected component of
$A$ is the intersection of $A$ with a connected component of
$F_{\eb,x}$.  By Mori's theorem, \cite{Mori}, this connected component
of $F_{\eb,x}$ is uniruled.  By Corollary \ref{cor-NK}, for every
rational curve $R$ in $F_{\eb,x}$, the union in $\mm$ of the curves
parameterized by $R$ is a surface that is the image of a generically
finite morphism from a $\CC\PP^1$-bundle over $R$.  Such a unirational
surface is itself a rational surface.  Moreover, the fiber $\CC\PP^1$
over each point of $R$ moves in a pencil of such curves on the surface
by varying the point of $R$.
\end{proof}

\begin{prop} \label{prop-mainA.5} \marpar{prop-mainA.5}
  For every minimal symplectically free class $\eb$, if $\eb$ is
  symplectically ruling, then $F_{\eb,x}$ is a singleton for every
  point $x\in \mm_\eb$.  The corresponding genus-$0$, $\eb$-curve
  indexed by $x\in \mm_\eb$ move in a pencil of rational curves on a
  rational surface. Thus $\mm$ is swept out by rational surfaces.
  Conversely, if $F_{\eb,x}$ is a singleton for a general
  $x\in \mm_{\eb}$, and if the corresponding genus-$0$, $\eb$-curve
  moves on a rational surface, then $\eb$ is symplectically ruling.
\end{prop}

\begin{proof}
  By Proposition \ref{prop-free}, the $\eb$-generic locus $\mm_\eb$ is
  a dense open whose complement is a proper closed analytic subspace
  of $\mm$.  Then $m_\eb$ equals $0$ and $f_\eb$ equals $1$ if and
  only if for every $x\in \mm_\eb$, the fiber $F_{\eb,x}$ is a
  singleton set.  In this case, the restriction of the universal curve
  over $\Kgnb{0,1}^o(\mm,\jj,\eb) = \mm_\eb$ is a family of $1$-cycles
  in $\mm$ parameterized by $\mm_\eb$.  This family of $1$-cycles
  defines a meromorphic map from $\mm$ to the Barlet space of $\mm$,
  $$
  R\to \text{Barlet}(\mm), \ \ \mm\supseteq \mm_\eb
  \xrightarrow{\epsilon} R^o \subseteq R.
  $$
  As the normalization of the image of a meromorphic map between
  compact, complex analytic spaces in Fujiki's class, the image $R$ is
  an irreducible, normal, compact, complex analytic space in Fujiki's
  class.  Up to blowing up $R$, assume that the MRC fibration of $R$
  is everywhere holomorphic.  By construction, the general fibers of
  $\epsilon$ are the genus-$0$, $\eb$-curves in $\mm$ parameterized by
  $\mm_\eb$.  In particular, the only $\text{ev}_1$-dominant curve
  classes that are contracted by $\epsilon$ are in the $\QQ$-span of
  $\eb$.

\mni
If there exists a symplectically pseudo-free class $\gamma$ that is
not in the $\QQ$-span of $\eb$, then the image class in $R$ is
$\text{ev}_1$-dominant.  By Part 1 of Theorem \ref{thm-simpFano}, the
K\"{a}hler manifold $R$ has a minimal free class (which is also
symplectically free).  For a general point of $R$, for a minimal free
curve $C$ containing that point, the inverse image of $C$ under
$\epsilon$ is a rational surface in $\mm$ on which the unique
genus-$0$, $\eb$-curve moves in a pencil (indexed by $C$).

\mni
Conversely, assume that a general fiber of $\epsilon$ moves in a
pencil on a rational surface.  The goal is to construct a
symplectically pseudo-free class $\gamma$ on $\mm$ that is
$\QQ$-linearly independent from $\eb$.  In fact, we will construct a
symplectically pseudo-free class $\gamma$ such that the pushforward
class $\gamma_R = \epsilon_*\gamma$ is nonzero.

\mni
Since $\mm$ is covered by rational surfaces, and every pair of points
of a rational surface is contained in a chain of rational curves in
that rational surface, not all rational curves of this surface can
equal fibers of $\phi$.  Thus, the image $R$ is uniruled by free
curves.  By the proof of Theorem \ref{thm-simpFano}, there exists a
minimal free curve class $\gamma_R$ on $R$.  Denote by $\oom_R$ a
K\"{a}hler class on $R$.  By the original proof of Theorem
\ref{thm-KR1}, for
$m_{\gamma,R} = \langle c_1(T^{1,0}_R),\gamma_R \rangle - 2$, we have
positivity of the Gromov-Witten invariant
$$
\langle \eta_R, [\oom_R]\smile [\oom_R], \dots, [\oom_R]\smile
[\oom_R] \rangle^{R,\oom_R}_{0,\gamma_R},
$$   
with $m_{\gamma,R}$ insertions of $[\oom_R]\smile [\oom_R]$.  More
precisely, when restricted to the fiber of the MRC fibration of $R$
containing a general point $q$, the class $[\oom_R]$ is an
$\RR_{>0}$-linear combination of ample $\ZZ$-divisor classes, so that
we can represent the restriction of $[\oom_R]\smile[\oom_R]$ to the
fiber as an $\RR_{>0}$-linear combination of Poincar\'{e} duals of
homology classes of moving, codimension-$2$ complex analytic
subvarieties of the fiber.  Thus, choosing these subvarieties
generically, the Gromov-Witten invariant is an $\RR_{>0}$-linear
combination of the homology classes of finitely many finite subsets of
$\Kgnb{0,m_{\eb,R}+1}((Q,\jj_R),\gamma_R)$, each of which equals the
finite set of all genus-$0$, $\gamma_R$-curves containing $q$ and
intersecting each of $m_{\gamma,R}$ specified codimension-$2$ closed,
analytic subvarieties of the MRC fiber of $q$.

\mni
For each genus-$0$, $\gamma_R$-curve $C$ in one of these finite sets,
the inverse image $\Sigma$ of $C$ under $\epsilon$ is a conic bundle
over $C$.  By Tsen's Theorem, this conic bundle admits holomorphic
sections.  For every class $\gamma$ of a section, the pushforward
under $\gamma$ equals $\gamma_R$.  Since the image of the MRC
fibration of $R$ is not uniruled, the composition of $\phi$ and the
MRC fibration of $R$ is an almost holomorphic fibration of $\mm$.  In
particular, every chain of genus-$0$ curves in $\mm$ that intersects
$\epsilon^{-1}(q)$ maps under $\phi$ to a chain of genus-$0$ curves in
the MRC fiber of $q$.  Thus, every chain $C'$ of genus-$0$ curves
whose total class equals $\gamma$ maps under $\epsilon$ to a chain of
genus-$0$ curves whose total class equals $\gamma_R$.  If the chain
$C'$ intersects $\epsilon^{-1}(q)$ as well as the $\epsilon$-inverse
images of the specified collection of codimension-$2$ complex analytic
subvarieties of the MRC fiber of $q$, then the image of $C'$ equals
$C$.  Thus, $C'$ is divisor in $\Sigma$.

\mni
For a divisor class $\gamma$ in $\Sigma$ of a complete linear system
$\PP^{n_\gamma}$ of moving section curves, for a general point
$x\in \epsilon^{-1}(q)$ and for $n_\gamma-1$ additional auxiliary
points $y_2,\dots,y_{n_\gamma}$, there is a unique curve $D$ in the
complete linear system that contains $x$ and the points $y_i$, and
$\epsilon$ maps $D$ isomorphically to $C$.  Of course the restriction
of $\oom$ to $\Sigma$ is an $\RR_{>0}$-linear combination of ample
divisor classes.  Thus, the restriction of $[\oom]\smile[\oom]$ to
$\Sigma$ is represented by an $\RR_{>0}$-linear combination of
Poincar\'{e} duals of homology classes of general points of $\Sigma$.
Altogether, we have positivity of the degree of the Gromov-Witten
invariant,
$$
\langle \eta_\mm, \epsilon^*([\oom_R]\smile
[\oom_R]),\dots,\epsilon^*([\oom_R]\smile[\oom_R]), [\oom]\smile
[\oom],\dots, [\oom]\smile [\oom]\rangle^{\mm,\oom}_{0,\gamma}
$$
with $m_{\gamma,R}$ of the insertions
$\phi^*([\oom_R]\smile [\oom_R])$ and with $n_\gamma-1$ of the
insertions $[\oom]\smile[\oom]$.  Indeed, the homology class of this
Gromov-Witten invariant equals a $\RR_{>0}$-linear combination of
finitely many homology classes of finite subsets of
$\Kgnb{0,m_{\gamma,R} + n_{\gamma}}((\mm,\jj),\gamma)$ parameterizing
genus-$0$, $\gamma$-curves $D$ in $\mm$ that (1) map to one of the
finitely many genus-$0$, $\gamma_R$-curves in $R$ that contains $x$,
itself mapping to $q$, that (2) intersects the inverse images under
$\epsilon$ of each of the $m_{\gamma,R}$ specified codimension-$2$
complex analytic subvarieties, and that (3) contains each of the
$n_\gamma-1$ general points $y_i$ in $\Sigma$.  Thus, $\gamma$ is a
symplectically pseudo-free class, and the pushforward
$\gamma_R=\epsilon_*\gamma$ is nonzero, so that $\gamma$ is
$\QQ$-linearly independent from $\eb$.  Therefore $\eb$ is
symplectically ruling.
\end{proof}

\mni
Part 3 of Theorem \ref{thm-surf} follows by the converse direction of
Proposition \ref{prop-mainA}.  Similarly, if $Z$ is symplectically
equivalent to $\mm$, then Part 2 of Theorem \ref{thm-surf} follows
from Proposition \ref{prop-mainA} and \ref{prop-mainA.5}.  Finally, if
$Z$ is symplectically deformation equivalent to $\mm$, then the cone
of symplectically free classes of $Z$ equals the cone of
symplectically free classes of $\mm$.  Assume that this cone is
nonzero, and assume that every indecomposable element of this common
cone is symplectically $2$-free or symplectically ruling (both of
these properties depend only on the symplectic deformation class).
Then for the symplectic form $\oom'$ on $Z$, there exists an
$\oom'$-minimal class in the symplectically free cone, by Theorem
\ref{thm-simpFano}.  Since this is a $\oom'$-minimal class, in
particular it is an indecomposable class of the cone.  Thus, this
$\oom'$-minimal class is either symplectically $2$-free or
symplectically ruling.  By Propositions \ref{prop-mainA} and
\ref{prop-mainA.5}, the K\"{a}hler manifold $Z$ is swept out by
rational surfaces.


\section{Proof of Theorem \ref{thm-PG}} \label{sec-GP2}  
\marpar{sec-GP2}

\mni
The following propositions adapt to projective homogeneous varieties
of arbitrary Picard rank some results valid in Picard rank $1$ from
\cite[Sections 14 and 15]{dJHS}.  For every connected, complex Lie
group $G$, the \textbf{solvable radical} $R_{\text{solv}}(G)$,
resp. the \textbf{unipotent radical} $R_{\text{uni}}(G)$, is the
maximal connected, normal, solvable, complex Lie subgroup of $G$,
resp. it is the kernel $R_{\text{uni}}(G)$ in $R_{\text{solv}}(G)$ of
the initial homomorphism of complex Lie groups from
$R_{\text{solv}}(G)$ to a group of multiplicative type,
$$
\chi: R_{\text{solv}}(G) \twoheadrightarrow M(G) \cong \CC^\times
\times \dots \times \CC^\times. 
$$
There exists a connected, complex Lie subgroup $L(G)$ such that the
induced homomorphism $L(G)\to G/R_{\text{uni}}(G)$ is an isomorphism.
This is a \textbf{Levi factor} of $G$, and $L(G)$ is unique up to
conjugation by $R_{\text{uni}}(G)$.  The group $G$ is
\textbf{semisimple} if $R_{\text{solv}}(G)$ is trivial.  This is true
for the commutator subgroup $[L(G),L(G)]$, which is called a
\textbf{semisimple factor} of $G$.

\mni
Let $G$ be a connected, simply connected, semisimple, complex Lie
group.  Let $B\subset G$ be a \textbf{Borel subgroup} $B$, i.e., a
connected, solvable, complex Lie subgroup of $G$ that is maximal with
these properties.  Let $T\subset B$ be a \textbf{maximal torus}, i.e.,
a connected, Abelian, complex Lie subgroup of $B$ of multiplicative
type $\CC^\times \times \dots \times \CC^\times$ that is maximal with
these properties.  Denote by $N_G(T)$ the normalizer of $T$ in $G$,
and denote by $W=W_{G,T}$ the associated \textbf{Weyl group}, i.e.,
the finite quotient subgroup $N_G(T)/T$.  Via conjugation, $W$ acts on
the set of all subgroups $H$ of $G$ that contain $T$.  Denote by $W_H$
the stabilizer subgroup of $H$ in $W$.  In particular, $W_B$ is the
trivial subgroup.

\mni
The set of \textbf{positive simple roots} is (naturally bijective to)
the subset $\Delta=\Delta_{G,T,B}$ of $W_{G,T}$ of all elements
$s \in W$ such that the complex Lie subgroup $H_s$ generated by $B$
and $s B s^{-1}$ has either $\textbf{SL}_2$ or $\textbf{PGL}_2$ for
its semisimple factor.  The positive simple roots give generators of
$W$ each having order $2$, and the minimal word length of an element
with respect to the generating set $\Delta$ is the \textbf{Coxeter
  length} of the element.

\mni
Let $P$ be a connected, proper, complex Lie subgroup of $G$ containing
$B$.  The quotient complex manifold $Y=G/P$ is projective.  For every
$W_P$-double coset $[w]\in W_P\ W / W_P$, the corresponding
\textbf{Schubert cell} is $(PwP)/P$ in $G/P$.  The closure of this
Schubert cell is a \textbf{Schubert variety}, and the homology class
is a \textbf{Schubert class} in $H_*(G/P;\ZZ)$.  The natural (left)
$G$-action on $Y$ induces a diagonal $G$-action on $Y\times Y$.

\begin{thm}[Bruhat Decomposition] \label{thm-Bruhat} \marpar{thm-Bruhat}
  There are finitely many $G$-orbits on $Y\times Y$, each of the form
  $G\cdot(E_w\times \{P/P\})$ for a unique $P$-orbit $E_{P,w} = PwP/P$
  in $Y=G/P$, indexed by $W_P$-double cosets
  $[w]\in W_P\backslash W/W_P$.  Each Schubert cell $E_{P,w}$ is a
  locally closed subvariety that is algebraically isomorphic to affine
  space $\CC^\ell$, where $\ell$ is the least Coxeter length of a
  representative of the double $W_P$-coset.  The homology classes of
  the Schubert varieties $\overline{E}_{P,w}$ form an additive basis
  for $H_*(G/P;\ZZ)$ that is self-dual under Poincar\'{e} duality.
\end{thm}

\mni
In particular, the $1$-dimensional Schubert varieties $\eb_i$ are in
bijection with the positive simple roots in
$\Delta_P := \Delta \setminus (\Delta \cap W_P)$, and the
corresponding Schubert classes form an additive basis for
$H_2(G/P;\ZZ)$. The Schubert classes of codimension-$1$ Schubert
varieties form a dual basis $D_i$ for $H^2(G/P;\ZZ)$.

\begin{prop} \label{prop-simp} \marpar{prop-simp}
  For $G/P$, the pseudo-effective cone equals the effective cone equals
  the base-point free cone equals the nef cone, and all of these equal
  the free $\ZZ_{\geq 0}$-semigroup with simplicial generators $D_i$.
  Dually, the Mori cone equals the movable cone equals the free
  $\ZZ_{\geq 0}$-semigroup with simplicial generators $\eb_i$.  The
  symplectically free cone equals the Mori cone, so that $G/P$ is
  integrally simplically Fano.  Also the $\eb_i$-generic locus equals
  all of $G/P$.
\end{prop}

\begin{proof}
  First of all, for every effective divisor (possibly empty), the
  $G$-translates of the divisor form a base-point free subset of the
  complete linear system.  Thus, the complete linear system of every
  effective divisor is base-point free, i.e., the associated invertible
  sheaf is globally generated.  Thus, the divisor class is nef.
  Dually, the Mori cone equals the movable cone.  In particular, every
  pseudo-effective divisor class has nonnegative intersections with
  every $1$-dimensional Schubert variety $\eb_i$.  Thus, the divisor
  class equals a nonnegative linear combination of the dual basis
  $D_i$.  Therefore the pseudo-effective cone equals the free
  $\ZZ_{\geq 0}$-semigroup generated by the classes $[D_i]$ which
  equals the base-point free cone.  Dually, the Mori cone is the free
  $\ZZ_{\geq 0}$-semigroup generated by the classes $\eb_i$.

\mni
Every symplectically free class is symplectically pseudo-free, which
is $\text{ev}_1$-dominant, and hence is an effective curve class.
Thus the symplectically free cone is contained in the Mori cone.  To
prove these cones are equal for $G/P$, it suffices to prove that every
$\eb_i$ is symplectically free.  By the $G$-translation action, every
effective curve is $\text{ev}_1$-dominant.  Since the classes $\eb_i$
are the simplicial generators of the Mori cone, these classes are
$\jj$-irreducible.  Thus, by Theorem \ref{thm-simpFano}, every $\eb_i$
is symplectically free.  Since the $\eb_i$-free locus is a dense
Zariski open that is $G$-invariant, it equals the entire homogeneous
variety $G/P$.
\end{proof}

\begin{prop} \label{prop-strong} \marpar{prop-strong}
  Every contraction of $G/P$ is $G$-equivariant, of the form
  $G/P\to G/Q$ for a parabolic subgroup $Q$ of $G$ containing $P$ that
  is uniquely determined by the classes $\eb_i$ contracted to points,
  or equivalently, by the positive simple roots in $\Delta_P\cap W_Q$.
  A class $\eb_i$ with $m_{\eb_i}$ equal to $0$ is symplectically
  ruling unless $G/P$ equals $\CC\PP^1$.  The contraction that
  contracts precisely the symplectically ruling classes $\eb_i$ is a
  Zariski locally trivial fiber bundle.  In particular, it is
  submersive.
\end{prop}

\begin{proof}
  The translation action of $G$ on $G/P$ induces a $G$-action on the
  projective linear system of any divisor. (In fact, since $G$ is
  assumed to be simply connected, this lifts uniquely to a linear
  action of $G$ on the vector space of global sections of the
  associated invertible sheaf.)  Thus, the projective contraction of
  $G$ given by the complete linear system is $G$-equivariant.  Thus
  the contraction is $G/P\to G/Q$ where $Q$ is the stabilizer group of
  the image under the contraction of the special point $x=P/P$ of
  $G/P$.  Every projective contraction is uniquely determined by the
  curve classes that are contracted, and these classes form an
  extremal face of the Mori cone.  Since the Mori cone is simplicial,
  this extremal face equals the span of the extremal rays contained in
  the cone, i.e., the span of the classes $\eb_i$ contracted.
  Finally, by the Bruhat decomposition for $G/Q$, the contracted
  curves $\eb_i$ correspond to the positive simple roots that are in
  $\Delta_P$ yet not in $\Delta_Q$, i.e., the intersection
  $\Delta_P\cap W_Q$.

\mni
For every class $\eb_i$ with $m_{\eb_i}=0$, since the $\eb_i$-free
locus equals all of $G/P$, the evaluation morphism,
$$
\text{ev}_1:\Kgnb{0,1}(G/P,\eb_i) \to G/P,
$$
is surjective and everywhere smooth of relative dimension
$m_{\eb_i}=0$, i.e., it is a finite unbranched covering.  Yet $G/P$ is
simply connected.  Thus the $\CC\PP^1$-bundle structure,
$$
\Kgnb{0,1}(G/P,\eb_i) \to \Kgnb{0,0}(G/P,\eb_i),
$$
induces a $\CC\PP^1$-bundle contraction of $G/P$, necessarily of the
form $G/P \to G/P_i$ where $P_i$ is the minimal parabolic containing
$P$ such that $\eb_i$ is not contained in $\Delta_{P_i}$.  In
particular, every curve whose class is in the extremal ray spanned by
$\eb_i$ is contracted in $G/P_i$, so that the $\eb_i$-curves are
precisely the fibers of this contraction.  Therefore, the degree
$f_{\eb_i}$ equals $1$; there is a unique $\eb_i$-curve containing
each point $x$ of $G/P$, namely the fiber of the contraction
containing $x$.  If $\text{dim}(G/P)\geq 2$, then also
$\text{dim}(G/P_i)\geq 1$.  Since $G/P_i$ is uniruled, by the proof of
Theorem \ref{thm-surf}, the class $\eb_i$ is symplectically ruling.

\mni
Let $D$ be the semi-ample divisor class that equals the sum of $D_j$
over all $j$ with $\eb_j$ not equal to a symplectically ruling class.
The $G$-equivariant contraction of the corresponding complete linear
system, say $\pi:Y\to \ol{Y}$, is of the form $G/P \to G/Q_i$. Thus it
is a Zariski locally trivial fiber bundle.  In particular, $H_2$ of
the fiber equals the kernel of the pushforward map,
$$
H_2(G/P;\ZZ) \to H_2(G/Q_i;\ZZ).
$$
By the construction of $G/Q_i$, this kernel is precisely the span of
those $\eb_i$ that are symplectically ruling.  Thus, $G/P$ is strongly
simplicially Fano.
\end{proof}

\begin{prop} \label{prop-2free} \marpar{prop-2free}
  Assume that $\text{dim}(G/P)\geq 2$.  Every $\eb_i$ that is not
  symplectically ruling is symplectically $2$-free.  In fact, both
  $\Kgnb{0,1}(G/P,\eb_i)$ and $\Kgnb{0,2}(G/P,\eb_i)$ have only
  finitely many $G$-orbits, both of these are connected, K\"{a}hler
  manifolds, and the first Chern class of each is represented by a
  nonzero, effective divisor.
\end{prop}

\begin{proof}
  Every $\eb_i$-curve is contained in a fiber $Z=gP_i/P$ of the
  associated extremal contraction $G/P\to G/P_i$ of $\eb_i$.  This
  fiber, in turn, is a projective homogeneous space of the semisimple
  Levi factor $G_i$ of $P_i$, which has Picard rank $1$; say $G_i/R_i$
  for a maximal parabolic subgroup $R_i$ of $G_i$.  The ample
  generator of the Picard group $D_i$ of $G_i/R_i$ has intersection
  pairing $1$ with $\eb_i$.  Since $D_i$ is effective, it is base-point
  free.  Since $D_i$ is ample, the associated contraction of $G_i/R_i$
  is finite.  Since it is $G_i$-equivariant, the contraction is
  everywhere smooth, i.e., it is a finite covering map.  However,
  every projective homogeneous space, including the image of the
  contraction, is simply connected.  Thus, the contraction is an
  isomorphism, i.e., $D_i$ is very ample.  Since $\eb_i$-curves have
  intersection number $1$ with $D_i$, these curves are mapped to lines
  under the projective embedding of the complete linear system of
  $D_i$.

\mni
Every line in projective space is uniquely determined by two distinct
points on the line.  Denote by $\sigma$ the diagonal section of the
projection $\Phi_k$ forgetting the $\text{k}^{\text{th}}$ marked
point, for $k=1,2,$,
$$
\Phi_k:\Kgnb{0,2}(G/P,\eb_i) \to \Kgnb{0,1}(G/P,\eb_i), \ \
\sigma:\Kgnb{0,1}(G/P,\eb_i) \to \Kgnb{0,2}(G/P,\eb_i).
$$
Denote by $\Kgn{0,2}(G/P,\eb_i)$ the open complement in
$\Kgnb{0,2}(G/P,\eb_i)$ of the image of $\sigma$.  Since every line is
uniquely determined by two distinct points on that line, the following
evaluation morphism is injective,
$$
\text{ev}_2:\Kgn{0,2}(G/P,\eb_i) \to G/P\times G/P.
$$
This morphism is also $G$-equivariant.

\mni
By the Bruhat decomposition, $G/P\times G/P$ has only finitely many
$G$-orbits.  Thus, also the $G$-invariant subvariety
$\Kgn{0,2}(G/P,\eb_i)$ has only finitely many $G$-orbits.  Since
$\Phi_k$ is $G$-equivariant and surjective, also
$\Kgnb{0,1}(G/P,\eb_i)$ has only finitely many $G$-orbits.  Thus, the
image of $\sigma$ has only finitely many $G$-orbits.  Altogether, both
$\Kgnb{0,2}(G/P,\eb_i)$ and $\Kgnb{0,1}(G/P,\eb_i)$ have only finitely
many $G$-orbits.

\mni
Since the $\eb_i$-free locus equals all of $G/P$, every $\eb_i$-curve
is free, so that both of $\Kgnb{0,2}(G/P,\eb_i)$ and
$\Kgnb{0,1}(G/P,\eb_i)$ are smooth, i.e., they are compact, K\"{a}hler
manifolds (possibly disconnected).  Also, the first evaluation
morphism
$$
\text{ev}_1\circ \Phi_1:\Kgnb{0,2}(G/P,\eb_i) \to
\Kgnb{0,1}(G/P,\eb_i) \to G/P, \ \ (C,q_1,q_2,u:C\to G/P) \mapsto u(q_1),
$$
is generically smooth (by Sard's Theorem), and $G$-equivariant, hence
everywhere smooth.  Since $G/P$ is simply connected, the restriction
of the evaluation morphism to every connected component of
$\Kgnb{0,2}(G/P,\eb_i)$ has smooth, connected fibers.  For that
connected component, denote by $F_{\eb_i,x}$ the fiber in that
connected component of the evaluation morphism over the special point
$x=P/P$.  This is a smooth, connected K\"{a}hler manifold with an
induced action of $P$.  The second evaluation morphism,
$$
\text{ev}_1\circ \Phi_2:\Kgnb{0,2}(G/P,\eb_i) \to
\Kgnb{0,1}(G/P,\eb_i) \to G/P, \ \ (C,q_1,q_2,u:C\to G/P) \mapsto u(q_1),
$$
maps $F_{\eb_i,x}$ $P$-equivariantly to an irreducible, closed,
$P$-invariant subset $C_{\eb_i,x}$ of $G/P$.  As with the Bruhat
decomposition on all of $G/P$, the homology of $C_{\eb_i,x}$ is a free
$\ZZ$-module on the Schubert classes of those Schubert varieties
$\overline{E}_{P,w}$ that happen to intersect $C_{\eb_i,x}$ (and thus
are contained in $C_{\eb_i,x}$).  In particular,
$H_2(C_{\eb_i,x};\ZZ)$ equals the free $\ZZ$-module on those
$1$-dimensional Schubert varieties $\eb_j$ that are contained in
$C_{\eb_i,x}$.  Since $C_{\eb_i,x}$ is a union of curves in the
homology class $\eb_i$, the image of the morphism,
$$
H_2(C_{\eb_i,x};\ZZ)\to H_2(G/P;\ZZ),
$$
contains the class of $\eb_i$.  Thus, among the free generators
$\eb_j$ of $H_2(C_{\eb_i,x};\ZZ)$, which also are linearly independent
in $H_2(G/P;\ZZ)$, there must be the class $\eb_i$.  Therefore, some
$\eb_i$-curve parameterized by $F_{\eb_i,x}$, i.e., containing
$x=P/P$, also contains a point of $\eb_i$ different from $x$.  Since
every $\eb_i$-curve is uniquely determined by two distinct points on
that curve, one of the $\eb_i$-curves parameterized by $F_{\eb_i,x}$
is the Schubert variety $\eb_i$.  Since every connected component of
the fiber over $x$ of $\text{ev}_1\circ \Phi_1$ contains the special
point $[\eb_i]$, there is only one connected component.  Therefore the
compact, K\"{a}hler manifolds $\Kgnb{0,0}(G/P,\eb_i)$,
$\Kgnb{0,1}(G/P,\eb_i)$, and $\Kgnb{0,2}(G/P,\eb_i)$ are each
connected.

\mni
Associated to the natural $G$-action on the connected, K\"{a}hler
manifold $\Kgnb{0,r}(G/P,\eb_i)$, there is an induced map from the Lie
algebra $\mf{g}$ of $G$ to the tangent bundle,
$$
\mu:\mf{g}\otimes_{\CC}\OO_{\Kgnb{0,r}(Y,\eb_i)} \to
T_{\Kgnb{0,r}(Y,\eb_i)}. 
$$
For each of $r=1,2,$ since there are only finitely many $G$-orbits
each of which is a locally closed subvariety, there is a unique open
$G$-orbit on $\Kgnb{0,r}(Y,\eb_i)$.  Thus, $\mu$ is generically
surjective.

\mni
For a general choice of a $\CC$-linear subspace $V$ of $\mf{g}$ of
dimension equal to the dimension of $\Kgnb{0,r}(Y,\eb_i)$, the
restriction $t_V$ of $t$ to $V\otimes_{\CC}\OO_{\Kgnb{0,r}(Y,\eb_i)}$
is generically an isomorphism.  Thus, the first Chern class of
$\Kgnb{0,r}(Y,\eb_i)$ is represented by the effective divisor that is
the degeneracy locus $\mc{D}$ of $t_V$.  Since $\Kgnb{0,r}(Y,\eb_i)$
has an almost homogeneous action of $G$, this is a projective variety
that is unirational.  Hence the tangent bundle is not trivial, so that
$\mc{D}$ is nonempty.

\mni
Consider the evaluation morphism,
$$
\text{ev}_1:\Kgnb{0,1}(Y,\eb_i) \to Y.
$$
The restriction of $\text{ev}_1$ to $\mc{D}$ is dominant.  By
Grothendieck's Generic Freeness Theorem, i.e., by ``generic
flatness'', this morphism is flat when restricted over a particular
dense Zariski open subset of $Y$.  Thus, for a general $x\in Y$, the
intersection of the fiber $F_{\eb_i,x}$ of $\text{ev}$ with $\mc{D}$
is a nonempty divisor in $F_{\eb_i,x}$ that represents the first Chern
class of $F_{\eb_i,i}$, i.e., $F_{\eb_i,x}$ is weakly Fano.  By the
converse direction of Theorem \ref{thm-surf}, the class $\eb_i$ is
symplectically $2$-free.  Since this holds for every simplicial
generator $\eb_i$ that is not symplectically ruling, the projective
homogeneous variety $G/P$ is symplectically $2$-free.
\end{proof}

\mni
Returning to the proof of Theorem \ref{thm-PG}, Part 1 follows from
Proposition \ref{prop-simp}.  Theorem \ref{thm-surf} and Proposition
\ref{prop-simp} give Part 2.  Finally, Theorem \ref{thm-surf} and
Proposition \ref{prop-2free} give Part 3.


\section{Proof of Proposition \ref{prop-fiber} and Theorem
  \ref{thm-CI}} \label{sec-CI} \marpar{sec-CI}

\begin{proof}[Proof of Proposition \ref{prop-fiber}]
  \textbf{1.}
  By Part 1 of Theorem \ref{thm-simpFano}, for every fiber type Fano
  manifold, the Mori cone $\snef{\oom}(\mm)$ equals $\nef{\jj}(\mm)$,
  and this Mori cone is dual to the K\"{a}hler cone.  Conversely,
  since $\snef{\oom}(\mm)$ is contained in the Mori cone, if this cone
  is dual to the K\"{a}hler cone, then it equals the Mori cone.  In
  particular, since $\snef{\oom}(\mm)$ is generated by classes on
  which $c_1(T^{1,0}_{\mm,\oom})$ has positive degree, this implies
  that $\mm$ is Fano.  The Mori cone of a Fano manifold is rational
  polyhedral.  By Part 1 of Theorem \ref{thm-simpFano}, the extremal
  rays are symplectically free.

\mni
\textbf{2.}
The only way that a fiber type Fano manifold can fail to be integrally
fiber type Fano is if the minimal free class $\eb_i$ generating an
extremal, fiber type contraction is a multiple $r>1$ of some integral
homology class $\gamma$.  But then this forces
$m_{\eb_i}+2 = \langle c_1(T^{1,0}_{\mm,\oom}),\eb_i \rangle$ to be
$r$ times the integer $\langle c_1(T^{1,0}_{\mm,\oom}),\gamma\rangle$,
which in turn is divisible by the Fano index (by Poincar\'{e}
duality).  Thus, the Fano pseudo-index is strictly greater than the
Fano index.  Conversely, if the fiber type Fano manifold is integrally
fiber type, then for the Poincar\'{e} dual homology class $\gamma$ to
the ample generator of the Picard $[H]$ group of the fiber, the class
$\gamma$ is represented by a free curve class (which is even
symplectically free).  The Fano index $i$ of the fiber is the unique
integer such that the restriction of $c_1(T^{1,0}_{\mm,\oom})$ equals
$i[H]$.  Since $\langle [H],\gamma\rangle$ equals $1$, the Fano
pseudo-index $\langle c_1(T^{1,0}_{\mm,\oom}),\gamma \rangle$ equals
the Fano index $i$.
\end{proof}

\mni
I am very grateful to the anonymous reader who recommended replacing
moving lemmas in an earlier draft by a more natural argument in the
proof of Theorem \ref{thm-CI}.  The argument used here to avoid moving
lemmas is inspired by \cite[Chapter 3]{Minoccheri}.

\begin{defn}\label{defn-CI} \marpar{defn-CI}
  Let $(Y,\OO_Y)$ be a compact, complex analytic space of pure
  dimension $n$.  Let $b$ be an integer with $1\leq b \leq n$.  A
  \textbf{flag of pseudodivisors} of length $b$ is an ordered
  $b$-tuple of nested, closed analytic subspaces,
  $$
  ((Z_1,\OO_{Z_1}),\dots,(Z_b,\OO_{Z_b})),
  $$
  i.e., closed immersions of complex analytic spaces,
  $$
  (e_j,e_j^{\#}):(Z_j,\OO_{Z_j}) \hookrightarrow (Z_{j-1},\OO_{Z_{j-1}}),
  $$
  for $j=1,\dots,b$, with $(Z_0,\OO_{Z_0})$ defined to equal
  $(Y,\OO_Y)$, together with an assignment to every $j=0,\dots,b-1$ of
  an ordered pair
  $$
  (\Ll_j,s_j:\OO_{Z_j}\to \Ll_j),
  $$
  of an invertible $\OO_{Z_j}$-module $\Ll_j$ together with a global
  section whose associated zero scheme in $(Z_j,\OO_{Z_j})$ equals
  $(Z_{j+1},\OO_{Z_{j+1}})$.  The flag is \textbf{reduced} if every
  $Z_j$ is a reduced complex analytic space.  The flag is
  \textbf{co-Stein} if for every $j=0,\dots,b-1$, the open subspace
  $Z_j^o := Z_j\setminus Z_{j+1}$ is a Stein analytic space.  The flag
  is \textbf{ample} if $\Ll_j$ is ample on $Z_j$ for every
  $j=0,\dots,b-1$.  The flag is \textbf{locally complete intersection}
  if every $(Z_j,\OO_{Z_j})$ is local complete intersection of pure
  dimension $n-j$.
\end{defn}

\begin{lem} \label{lem-CI} \marpar{lem-CI}
  Every ample flag of pseudodivisors is co-Stein.  For every ample
  flag of pseudodivisors, if $(Y,\OO_Y)$ is a local complete
  intersection, then the flag is locally complete intersection if and
  only if $(Z_b,\OO_{Z_b})$ has pure dimension $n-b$.  In this case,
  for every $c=0,\dots,n-b+1$, if the singular locus of
  $(Z_b,\OO_{Z_b})$ has dimension $\leq n-b-c$, then for every
  $i=0,\dots,b$, also the singular locus of $(Z_j,\OO_{Z_j})$ has
  dimension $\leq n-j-c$.  In particular, if $(Z_b,\OO_{Z_b})$ is
  reduced, then the flag is reduced.
\end{lem}

\begin{proof}
  For each $j=0,\dots,b-1$, denote by $Z'_j$ the union of all
  irreducible components of $Z_j$ that are not completely contained in
  $Z_{j+1}$.  By hypothesis, $Z'_j$ is a complex projective algebraic
  variety and $Z'_j\cap Z_{j+1}$ is an ample hypersurface in $Z'_j$.
  Thus, the open complement $Z_j^o = Z'_j\setminus (Z'_j\cap Z_{j+1})$
  is a Stein analytic space.  Thus, the flag is co-Stein.

\mni
The local complete intersection result is proved by induction on $b$.
When $b=0$, there is nothing to prove.  Thus, by way of induction,
assume that $b\geq 1$ and assume the result is true for $b-1$.

\mni
For $j=1,\dots,b$, by the Principal Ideal Theorem, every irreducible
component of $Z_j$ has dimension $\geq n-j$.  For $j=1,\dots,b-1$, if
an irreducible component of $Z_j$ has dimension $d_j \geq 1$, then the
zero scheme $Z_{j+1}$ is nonempty and has dimension $\geq d_j-1$.
Thus, since $n-(b-1)$ is $\geq 1$, if there is any irreducible
component of $Z_{b-1}$ that has dimension $d_{b-1} \geq n-(b-1) + 1$,
then also $Z_b$ has an irreducible component of dimension
$\geq d_{b-1} -1 \geq n-b+1$.  This contradicts the hypothesis that
$Z_b$ has pure dimension $n-b$.  Therefore, by contradiction, also
$Z_{b-1}$ has pure dimension $n-(b-1)$.  By the induction hypothesis,
for every $j=1,\dots,b$, the zero scheme $Z_j$ has pure dimension
$n-j$.

\mni
Similarly, if the singular locus of $Z_{b-1}$ has dimension
$\geq n-(b-1)-c+1$, which is $\geq 2$ since $c\leq n-b+1$, then the
intersection of $Z_b$ with that singular locus is nonempty and has
dimension $\geq n-b-c+1 > n-b-c$.  That contradicts the hypothesis
that the singular locus of $(Z_b,\OO_{Z_b})$ has dimension
$\leq n-b-c$.  Therefore, again by induction, for every $j=1,\dots,b$
the zero scheme of $Z_j$ has dimension $\leq n-j-c$.  In particular,
setting $c$ equal to $n-b+1$, if $(Z_b,\OO_{Z_b})$ is reduced, then
also every $(Z_i,\OO_{Z_i})$ is reduced.
\end{proof}

\begin{thm}\cite[Theorems 3.2.1, 3.4.1]{HammLe} \label{thm-LHT}
  \marpar{thm-LHT}
  For every compact, complex analytic space $(Y,\OO_Y)$ of pure
  dimension $n$, for every reduced, co-Stein flag of $b\leq n$
  pseudodivisors in $(Y,\OO_Y)$, for every $j=1,\dots,b$, the relative
  homotopy group $\pi_i(Z_{j-1},Z_j)$ is zero for $i\leq n-j$.  By the
  Hurewicz Theorem, also $H_i(Z_{j-1},Z_j;\ZZ)$ is zero for
  $i\leq n-j$.  Thus, each induced homomorphism,
  $$
  \pi_i(Z_j) \to \pi_i(Z_{j-1}), \ \
  H_i(Z_j;\ZZ) \to H_i(Z_{j-1};\ZZ)
  $$
  is an isomorphism for every $i<n-j$, and is an epimorphism for $i=n-j$.
\end{thm}

\begin{proof}
  By hypothesis, every inclusion $Z_{j}\hookrightarrow Z_{j-1}$ is a
  hypersurface whose open complement is Stein and local complete
  intersection of pure dimension $n-j+1$.  Thus, the result follows
  from \cite[Theorems 3.2.1 and 3.4.1]{HammLe}.
\end{proof}

\mni
Let $(\nn,\jj_{\nn},\oom_{\nn})$ be a compact K\"{a}hler manifold that
has a submersive contraction of ruling classes,
$$
\pi:\nn \to \nn'.
$$

\begin{prop} \label{prop-Serre} \marpar{prop-Serre}
  The fibers $Y_q$ of $\pi$ are connected and simply connected.  Thus,
  for every submanifold $\mm'$ of $\nn'$ and its inverse image
  $\mm = \pi^{-1}(\mm')$, the Serre spectral sequence for $H_2$
  degenerates to a short exact sequence,
  $$
  0 \to H_2(\nn_q;\ZZ) \to H_2(\mm;\ZZ) \to H_2(\mm';\ZZ)\to 0.
  $$
\end{prop}

\begin{proof}
  Since $\pi$ is a contraction, the fibers are connected.  By
  hypothesis, the morphism $\pi$ is everywhere submersive.  Thus,
  $\nn_q$ is a connected, compact, K\"{a}hler manifold.

\mni
Now we repeat the proof of Part 1 of Theorem \ref{thm-surf}.  The
K\"{a}hler manifold $\nn_q$ has a rational quotient.  Either $\nn_q$
is rationally connected or the the pullback of a K\"{a}hler form from
the target of the rational quotient is a nonzero element of $H^2$.
Since $H_2(\nn_q;\ZZ)$ is generated by classes of free rational
curves, each of which is contracted by the rational quotient, the
pairing of this element of $H^2$ is zero with all of $H_2(\nn_q;\ZZ)$,
contradicting Poincar\'{e} duality.  Thus, the rational quotient is a
point, i.e., $\nn_q$ is rationally connected.  Every connected,
compact, K\"{a}hler manifold that is rationally connected is also
simply connected, \cite{C91}.  Therefore, every fiber $\nn_q$ is
simply connected.

\mni
Since the fibers of $\pi$ are connected and simply connected, the
portion of the Serre spectral sequence for the homology $\mm$
computing $H_2(\mm;\ZZ)$ degenerates to the stated short exact
sequence.
\end{proof}

\mni
As in Definition \ref{defn-CI}, let
$((Z_1,\OO_{Z_i}),\dots,(Z_c,\OO_{Z_c}))$ and
$((\Ll_0,s_0),\dots,(\Ll_{c-1},s_{c-1}))$ be a flag of pseudodivisors
in a K\"{a}hler manifold $\nn'$.  Denote $Z_c$ by $\mm'$.

\begin{prop} \label{prop-Lefschetz} \marpar{prop-Lefschetz}
  If the flag of pseudodivisors is ample and locally complete
  intersection, and if $\mm'$ a reduced complex analytic space with
  $\text{dim}(\mm') \geq 3$, then the pushforward map on homology is
  an isomorphism,
  $$
  H_2(\mm';\ZZ) \xrightarrow{\cong} H_2(\nn';\ZZ).
  $$
  For a submersive contraction $\pi$ as in Proposition
  \ref{prop-Serre}, for $\mm:= \pi^{-1}(\mm')$, also the following
  pushforward map is an isomorphism,
  $$
  H_2(\mm;\ZZ) \xrightarrow{\cong} H_2(\nn;\ZZ).
  $$
\end{prop}

\begin{proof}
  The first isomorphism follows from the Lefschetz hyperplane theorem,
  Theorem \ref{thm-LHT}.  Thus, in the commutative diagram of short
  exact sequences,
  $$
  \begin{CD}
    0 @>>> H_2(F;\ZZ) @>>> H_2(\mm;\ZZ) @>>> H_2(\mm';\ZZ) @>>> 0 \\
   & &  @VVV  @VVV  @VVV \\
      0 @>>> H_2(F;\ZZ) @>>> H_2(\nn;\ZZ) @>>> H_2(\nn';\ZZ) @>>> 0 \\
  \end{CD}
  $$
  the first and the third vertical homomorphisms are isomorphisms.  By
  the Snake Lemma, also the middle vertical homomorphism is an
  isomorphism.
\end{proof}

\begin{prop} \label{prop-pseudofree} \marpar{prop-pseudofree}
  For every compact K\"{a}hler manifold $\nn$, for every codimension
  $c$ complex submanifold $\mm$ that is a complete intersection
  $\nn_1\cap \dots \cap \nn_c$ of nef divisors $\nn_j$ in $\nn$, every
  free class in $\mm$ pushes forward to a free class in $\nn$.  In
  particular, every decomposition of a $\jj$-irreducible
  $\text{ev}$-dominant class on $\nn$ as a (possibly zero) sum of
  genus-$0$ classes and the pushforward of an $\text{ev}$-dominant
  class on $\mm$ is a trivial decomposition, i.e., the
  $\text{ev}$-dominant class on $\nn$ equals the pushforward of the
  $\text{ev}$-dominant class on $\mm$.
\end{prop}

\begin{proof}
  This is well-known.  The normal bundle of $\mm$ in $\nn$ equals the
  restriction to $\mm$ of the direct sum of invertible sheaves
  $\bigoplus_j \OO_{\nn}(\ul{\nn}_j)$.  Since each of the divisors
  $\nn_j$ is nef, the restriction of each of these invertible sheaves
  to every irreducible curve has nonnegative degree.  Thus, the
  restriction to a free curve in $\mm$ of the normal bundle is a
  direct sum of invertible sheaves of nonnegative degrees, i.e., it is
  semi-positive.  Since the restrictions to the free curve of both the
  tangent bundle of $\mm$ and the normal bundle are semi-positive, also
  the restriction of the tangent bundle of $\nn$ is semi-positive.
  Thus, the curve is free in $\nn$.

  \mni
  Every $\text{ev}$-dominant class on $\mm$ is a sum of a free class
  on $\mm$ and a (possibly zero) sum of genus $0$ curves.  By the
  previous paragraph, the pushforward to $\nn$ of the free class is
  free on $\nn$, hence $\text{ev}$-dominant.  Thus, if the class on
  $\nn$ is a $\jj$-irreducible $\text{ev}$-dominant class, then it
  equals the pushforward of the free class on $\mm$.
\end{proof}

\mni
Let $(\nn,\jj_{\nn},\oom_{\nn})$ be a K\"{a}hler manifold.  As in
Definition \ref{defn-CI}, let
$((Z_1,\OO_{Z_i}),\dots,(Z_c,\OO_{Z_c}))$ and
$((\Ll_0,s_0),\dots,(\Ll_{c-1},s_{c-1}))$ be a flag of pseudodivisors
in $\nn$.  Denote $Z_c$ by $\mm$.  Let $\eb$ be a curve class on $\nn$
that is in the pushforward of $H_2(\mm;\ZZ)$ so that every integer,
$$
m_j := \langle c_1(\Ll_j), \eb \rangle,
$$
is defined.  Define $m_\eb(Y,\mm):=\sum_{j=1}^c m_j$, and define
$m^+_\eb(Y,\mm):=\sum_{j=1}^c \max(m_j,0)$.  Recall the notation for
the evaluation morphism,
$$
\text{ev}_{0,1,\eb}:\Kgnb{0,1}(\nn,\eb) \to \nn.
$$

\begin{lem} \label{lem-pseudo} \marpar{lem-pseudo}
  Assume that the fiber of $\text{ev}_{0,1,\eb}$ over a point $x$ of
  $\mm$ is disjoint from the boundary.  For every $j=1,\dots,c$, if
  $m_j \leq 0$, then the fiber $F^{Z_j}_{\eb,x}$ for $Z_j$ equals the
  fiber $F^{Z_{j-1}}_{\eb,x}$ for $Z_{j-1}$.  If $m_j \geq 0$, then
  $F^{Z_j}_{\eb,x}$ arises from an ample flag of $m_j$ pseudodivisors
  in $F^{Z_{j-1}}_{\eb,x}$ whose sequence of invertible sheaves equals
  $$
  (\psi^{\otimes 1},\psi^{\otimes 2},\dots,\psi^{\otimes m_j}).
  $$
  If $m^+_\eb(\nn,\mm) \leq m_\eb(\nn)$, then $F^{\mm}_{\eb,x}$ is
  nonempty, and every irreducible component has dimension
  $\geq m_\eb(\nn)-m^+_\eb(\nn,\mm)$.
\end{lem}

\begin{proof}
  Denote $F^{Z_{j-1}}_{\eb,x}$ by $\So$, denote the universal family
  of curves over $\So$ by $\pi:\Cc\to \So$.  Denote the constant
  section of $\pi$ by $\sigma:\So \to \Cc$.  Denote the evaluation
  morphism by
  $$
  u:\Cc \to Z_{j-1}.
  $$
  Then $u^*s_j$ is a global section of $u^*\Ll_j$ that evaluates to
  zero under restriction to $\sigma(\So)$.  By Corollary
  \ref{cor-psi1}, $u^*\Ll_j$ is isomorphic to
  $\mc{F}_{m_j} = [\pi^*\psi(\ul{\sigma})]^{\otimes m_j}$.  Thus,
  $u^*s_j$ is a global section of the pushforward
  $\mc{E}_{m_j} = \pi_*\mc{F}_{m_j}$ that evaluates to zero in the
  structure sheaf.  By the proof of Corollary \ref{cor-psi}, the
  kernel of evaluation is an ample locally free sheaf of rank $m_j$,
  and it admits a filtration by locally free sheaves whose associated
  graded sheaves are $(\psi^{\otimes 1},\dots,\psi^{\otimes m_j})$.
  The images of $u^*s_j$ in this associated graded sheaves define a
  sequence of pseudodivisors whose common zero locus is
  $F^{Z_j}_{\eb,x}$.

\mni
Concatenating these sequences of pseudodivisors, and using the fact
that the zero scheme of an ample invertible sheaf on a proper analytic
space of dimension $\geq 1$ is nonempty, if $m_\eb(\nn,\mm)$ is less
than the dimension $m_\eb(\nn)$ of $F^{\nn}_{\eb,x}$, then also
$F^\mm_{\eb,x}$ is nonempty of dimension
$\geq m_\eb(\nn) - m_\eb(\nn,\mm)$.
\end{proof}

\begin{prop} \label{prop-chern2} \marpar{prop-chern2}
  With hypotheses as in Lemma \ref{lem-pseudo}, also assume that $\mm$
  is a codimension-$c$ intersection of nef hypersurfaces $\nn_j$ in
  $\nn$, so that also the hypotheses of Proposition
  \ref{prop-pseudofree} hold and $m_\eb^{+}(\nn,\mm)$ equals
  $m_\eb(\nn,\mm)$.

\noindent
\textbf{1.}
The subspace $F^{\mm}_{\eb,x}$ is nonempty for general $x\in \mm$ if
and only if $m_{\eb}(\nn,\mm) \leq m_{\eb}(\nn)$.

\noindent
\textbf{2.}
In this case, for $x\in \mm$ general, $F^{\mm}_{\eb,x}$ is smooth of
pure dimension $m_{\eb}(\mm) = m_{\eb}(\nn) - m_{\eb}(\nn,\mm)$ and
parameterizes only free, irreducible maps to $\mm$.

\noindent
\textbf{3.}
Considered as a closed subspace of $F^{\nn}_{\eb,x}$, the manifold
$F^{\mm}_{\eb,x}$ arises from an ample, locally complete intersection
flag of pseudodivisors as in Lemma \ref{lem-pseudo}.

\noindent
\textbf{4.}
The pushforward to $F^{\nn}_{\eb,x}$ of the pullback to
$F^{\mm}_{\eb,x}$ of each cycle class $w$ equals
$w\cdot \prod_j(m_j!)  c_1(\psi)^{m_\eb(\nn,\mm)}$.  Also, the
pushforward to $F^{\nn}_{\eb,x}$ of the first Chern class of the
normal sheaf of $F^{\mm}_{\eb,x}$ in $F^{\nn}_{\eb,x}$ equals
$$
\lt(\sum_{j=1}^c(m_j(1+m_j)/2)c_1(\psi)\rt) \cdot
\prod_{j=1}^c(m_j!) c_1(\psi)^{m_\eb(\nn,\mm)}.
$$
\end{prop}

\begin{proof}
  If $m_{\eb}(\nn,\mm) \leq m_{\eb}(\nn)$, then by Lemma
  \ref{lem-pseudo}, $F^{\mm}_{\eb,x}$ is nonempty for every
  $x\in \mm$.  Conversely, assume that $F^{\mm}_{\eb,x}$ is nonempty
  for general $x\in \mm$.  Since $F^{\nn}_{\eb,x}$ parameterizes only
  irreducible curves, the same holds for the subspace
  $F^{\mm}_{\eb,x}$.  Thus, the curves parameterized by
  $F^{\mm}_{\eb,x}$ are all free, irreducible curves.  Therefore
  $F^{\mm}_{\eb,x}$ is a manifold of pure dimension equal to the
  expected dimension
  $m_\eb(\mm)=\langle c_1(T^{1,0}_{\mm,\jj}),\eb \rangle -2$.  By
  adjunction applied to the flag of pseudodivisors defining $\mm$ in
  $\nn$, $m_\eb(\mm)$ equals $m_\eb(\nn)-m_\eb(\nn,\mm)$.  Since this
  integer is nonnegative, $m_\eb(\nn,\mm) \leq m_\eb(\nn)$.

\mni
By Proposition \ref{prop-pseudofree}, the subspace $F^{\mm}_{\eb,x}$
is contained in the maximal open subspace of $F^{\nn}_{\eb,x}$ that
parameterizes free curves, and this open has pure dimension equal to
$m_{\eb}(\nn)$.  By Lemma \ref{lem-pseudo}, the ample flag of
pseudodivisors in $F^{\nn}_{\eb,x}$ from Lemma \ref{lem-pseudo} is
locally complete intersection, i.e., each pseudo-divisor is an honest
divisor.  Pullback of cycles to a divisor follows by pushforward is
simply cup product with the first Chern class of the divisor.
Iterating through all divisors in the flag gives the formula for the
pushforward to $F^{\nn}_{\eb,x}$ of the pullback of cycles to
$F^{\mm}_{\eb,x}$.  Finally, since $F^{\mm}_{\eb,x}$ arises from a
locally complete intersection flag of divisors, the conormal sheaf has
an associated filtration whose associated graded sheaves are the
invertible sheaves that equal the restriction of the duals of the
invertible sheaves from the flag of pseudodivisors.  By the Whitney
sum formula, the first Chern class of the conormal sheaf equals the
sum of all of the first Chern classes of these restrictions.  This
gives the second formula.
\end{proof}

\begin{proof}[Proof of Theorem \ref{thm-CI}.] \textbf{1.}
  By Proposition \ref{prop-Lefschetz}, the pushforward map on $H_2$ is
  an isomorphism.

  \mni \textbf{2.}
  First, by Proposition \ref{prop-pseudofree}, the pushforward to
  $\nn$ of the free cone of $\mm$ is contained in the free cone of
  $\nn$.  The goal is to characterize when the pushforward cone is the
  entire free cone of $\nn$.

\mni
Let $\eb_i\in H_2(\nn,\ZZ)$ be a primitive generator of an extremal
ray of $\nef{\jj}(\nn)_{\RR}$.  By hypothesis, $\eb_i$ is free:
$\eb_i$ is one of the extremal rays of the free cone.  Since $\eb_i$
is a primitive generator of an extremal ray, it is also
$\jj$-irreducible in the Mori cone.  Thus, every $\eb_i$-curve in
$\nn$ is irreducible, i.e., the hypothesis of Lemma \ref{lem-pseudo}
holds.  By Proposition \ref{prop-chern2}, this class is the
pushforward of a free curve class from $\mm$ if and only if
$m_{\eb_i}(\nn,\mm) \leq m_{\eb_i}(\nn)$, in which case $m_{\eb}(\mm)$
equals $m_{\eb_i}(\nn) - m_{\eb_i}(\nn,\mm)$.  Thus, the pushforward
map identifies free cones if and only if every
$m_{\eb_i}(\nn,\mm)\leq m_{\eb_i}(\nn)$.  In this case, by Part 4 of
Proposition \ref{prop-chern2}, also $f_{\eb_i}(\mm)$ equals
$f_{\eb_i}(\nn)\cdot f_{\eb_i}(\nn,\mm)$.  Finally, if
$m_{\eb_i}(\mm)$ is positive, then Part 4 of Proposition
\ref{prop-chern2} also gives the identities,
$$
q_{\eb_i}(\mm) = q_{\eb_i}(\nn) - q_{\eb_i}(\nn,\mm), \ \
s_{\eb_i}(\mm) = f_{\eb_i}(\nn,\mm)\cdot \left(
  s_{\eb_i}(\nn)-s_{\eb_i}(\nn,\mm) \right).
$$

\mni
\textbf{3.}
This follows from the identities above and Part 2 of Theorem
\ref{thm-surf}.
\end{proof}

\mni
\textbf{Acknowledgments.}
Part of this work was supported by NSF Grants DMS-0846972 and
DMS-1405709, as well as a Simons Foundation Fellowship. I am grateful
to Yujiro Kawamata, Gang Tian, and Chenyang Xu for inviting me to
speak about this work during the wonderful conference, ``Stability,
Boundedness and Fano Varieties''.  I am grateful to Mark McLean, Zhiyu
Tian, and especially to Aleksey Zinger for patiently explaining the
fundamentals of Gromov-Witten theory in symplectic geometry and giving
me useful feedback on the exposition.  I thank Ljudmila Kamenova for
help about Stein spaces.  I thank Cinzia Casagrande who helped me with
Question \ref{ques-simpFano} and the references on this question.  I
thank the referees for their feedback.  I particularly thank the
referee who pointed out that the positive hypersurfaces in Theorem
\ref{thm-CI} need not be moving.

\bibliography{my}
\bibliographystyle{alpha}

\end{document}